\documentclass{article}

\usepackage{arxiv}

\usepackage[utf8]{inputenc} 
\usepackage[T1]{fontenc}    
\usepackage{hyperref}       
\usepackage{url}            
\usepackage{booktabs}       
\usepackage{amsfonts}       
\usepackage{nicefrac}       
\usepackage{microtype}      
\usepackage{lipsum}		
\usepackage{graphicx}
\usepackage{natbib}
\usepackage{doi}
\usepackage{amsmath, amsthm, amssymb, algorithm, algpseudocode}
\usepackage[table]{xcolor}
\usepackage{tikz}
\newcommand{\seq}[1]{\left<#1\right>}
\newcommand{\norm}[1]{\left\Vert#1\right\Vert}
\newcommand{\abs}[1]{\left\vert#1\right\vert}
\newcommand{\set}[1]{\left\{#1\right\}}
\newcommand{\bra}[1]{\left(#1\right)}
\newcommand{\sbra}[1]{\left[#1\right]}

\newcommand{\N}{\mathbb N}

\newcommand{\h}{\mathcal{H}}
\newcommand{\hh}{\mathcal{H}\oplus\mathcal{H}}
\newcommand{\bh}{\mathcal{B}(\mathcal{H})}
\newcommand{\bhh}{\mathcal{L}(\mathcal{H}\oplus\mathcal{H})}

\newcommand{\x}{\mathbf{x}}
\newcommand{\Y}{\mathbb{Y}}
\renewcommand{\c}{\mathbb{C}}
\newcommand{\T}{\mathbb{T}}
\newcommand{\A}{\mathbb{A}}
\newcommand{\B}{\mathbb{B}}
\renewcommand{\S}{\mathbb{S}}
\renewcommand{\L}{\mathcal{B}}
\newtheorem{theorem}{Theorem}[section]
\newtheorem{lemma}[theorem]{Lemma}

\newtheorem{corollary}[theorem]{Corollary}

\newtheorem{example}[theorem]{Example}

\newtheorem{remark}[theorem]{Remark}

\newcommand\mystyle{\everymath{\displaystyle}}
\mystyle
\title{Improved Upper Bounds for Numerical Radius Inequalities of Operator Matrices and Their Applications}

\author{\href{https://orcid.org/0000-0002-3816-5287}{\includegraphics[scale=0.06]{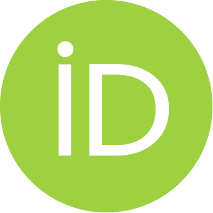}\hspace{1mm}M.H.M.~Rashid}\thanks{Corresponding Author} \\
	Department of Mathematics\&Statistics\\Faculty of Science P.O.Box(7)\\
	Mutah University University\\
	Mutah-Jordan \\
	\texttt{mrash@mutah.edu.jo}
}

\renewcommand{\shorttitle}{\textit{arXiv} Template}

\hypersetup{
pdftitle={Improved Upper Bounds for Numerical Radius Inequalities of Operator Matrices and Their Applications},
pdfsubject={q-bio.NC, q-bio.QM},
pdfauthor={M.H.M.Rashid},
pdfkeywords={Numerical radius; Operator norm Inequality; Operator matrices},
}

\begin{document}
\maketitle

\begin{abstract}
	 This study presents new upper bounds for the numerical radii of operator matrices, with a focus on $n \times n$ and $2 \times 2$ block matrices acting on Hilbert space direct sums. By employing techniques such as the H\"older-McCarthy inequality, Jensen's inequality, Bohr's inequality, and extensions of the Buzano inequality, we derive improved estimates that refine classical results by Kittaneh, El-Haddad, Dragomir, Buzano, and others. Our main contributions include bounds for general operator matrices, specific results for off-diagonal matrices, and inequalities for sums and products of operators. Through detailed comparisons and special cases, we demonstrate that our results enhance existing numerical radius inequalities. A key feature of our work is the use of parameter-dependent refinements with constants derived from extended Buzano-type inequalities. The applications of these results span quantum mechanics, integro-differential equations, and fractional calculus, providing sharper tools for stability analysis and numerical approximation in mathematical physics.
\end{abstract}


\section{Introduction}
Let $\bh$ be the $C^*$-algebra of all bounded linear operators
on the complex Hilbert space $\h$.  An operator $A\in\bh$ is said to be {\it positive} if $\seq{Ax,x}\geq 0$ holds for all $x\in\h$
and we write $A\geq 0$.
For $A\in\bh$, the numerical radius $\omega(\cdot)$ and the
usual operator norm $\norm{\cdot}$ are, respectively, defined by
\begin{equation*}
  \omega(A)=\sup_{\norm{x}=1}\abs{\seq{Ax,x}}\,\,\mbox{and}\,\, \norm{A}=\sup_{\norm{x}=1}\norm{Ax}.
\end{equation*}
It is clear that $\omega(\cdot)$ defines a norm on $\bh$. Moreover, it is known that $\omega(\cdot)$ is
equivalent to the usual operator norm $\norm{\cdot}$ on $\bh$ and with the following two
sided inequality
\begin{equation}\label{Ineq1.1}
  \frac{1}{2}\norm{A}\leq \omega(A)\leq \norm{A}\,\,\mbox{for all $A\in\bh$}.
\end{equation}
Consequently, if $A^2=0$, then
\begin{equation}\label{IneqB}
  \omega(A)=\frac{1}{2}\norm{A}.
\end{equation}
An important property for the numerical radius is the power inequality, which
says that
\begin{equation*}
  \omega(A^n)\leq \omega^n(A) \,\,\mbox{for all $n\in\N$ and $A\in\bh$}.
\end{equation*}
For basic properties of the numerical radius, we refer to \cite{Halmos}.\\
\indent In \cite{Kittaneh}, Kittaneh provided refinements of the bounds in (\ref{Ineq1.1}) by showing that
\begin{equation}\label{Ineq1.2}
  \frac{1}{4}\norm{|A|^2+|A^*|^2}\leq \omega^2(A)\leq \frac{1}{2}\norm{|A|^2+|A^*|^2} \,\,\mbox{for all $A\in\bh$},
\end{equation}
{where $|A|=\sqrt{A^*A}$ is the absolute value of $A$.}
In \cite{Had-Kit}, El-Haddad and Kittaneh provided a generalization for the second inequality
in (\ref{Ineq1.2}) by showing that
\begin{equation}\label{Ineq1.3}
  \omega^{2r}(A)\leq \frac{1}{2}\norm{|A|^{2r}+|A^*|^{2r}} \,\,\mbox{for all $r\geq 1$ and $A\in\bh$}.
\end{equation}
In \cite{Drag4}, Dragomir presented an important upper bound for the numerical radii of
products of two operators by showing that if $T, S\in B(H)$ and $r \geq 1$, then
\begin{equation}\label{Ineq1.4}
  \omega^{r}(S^*T)\leq \frac{1}{2}\norm{|T|^{2r}+|S|^{2r}}.
\end{equation}
Let $\h_1,\h_2,\cdots,\h_n$ be Hilbert spaces. If $\h=\bigoplus_{j=1}^{n}\h_j$ and $S\in\bh$, then the operator
$S$ can be represented as an $n\times n$ operator matrix, i.e., $S=[S_{ij} ]_{n\times n}$ with $S_{ij }\in \L(\h_j,\h_i)$,
the space of all bounded linear operators from $\h_j$ to $\h_i$. Operator matrices provide a useful
tool for studying Hilbert space operators, which have been extensively studied in the literature.
The block-norm matrix $\hat{S}$ associated with an operator matrix $S=[S_{ij} ]_{n\times n}$  is defined by
{$\hat{S}=\sbra{\norm{S_{ij}}}_{n\times n}$} which is an $n\times n$ non-negative matrix. Hou et al. \cite{Hou} established some estimate
for the numerical radii, operator norm and spectral radii of an $n\times n$ operator matrix $S=\sbra{S_{ij}}$.
In particular, they showed that if $S=[S_{ij} ]_{n\times n}$ is an operator matrix and $\hat{S}=\sbra{\norm{S_{ij}}}_{n\times n}$  is
its block-norm matrix, then
\begin{equation}\label{Ineq1.5}
  \omega(S)\leq \omega(\hat{S}), \quad \norm{S}\leq \norm{\hat{S}},\quad\mbox{and}\,\, r(S)\leq r(\hat{S}),
\end{equation}
where $r(S)$  denotes the spectral radius of $S$.  Bani-Domi and Kittaneh \cite{Dom-kit1} presented a new
numerical radius inequality, which is an improvement of the inequality (\ref{Ineq1.5}). They proved that
if $S=[S_{ij}]$ {is} an operator matrix in $\L\bra{\bigoplus_{j=1}^{n}\h_j}$, then
\begin{equation}\label{Ineq1.6}
  \omega(S)\leq \omega\bra{[s_{ij}]},\,\,\mbox{where}\, s_{ij}=\left\{
                \begin{array}{ll}
                  \frac{1}{2}\bra{\norm{S_{ii}}+\norm{S_{ii}^{2}}^{1/2}}, & \hbox{if $i=j$;} \\
                  \norm{S_{ij}}, & \hbox{if $i\neq j$.}
                \end{array}
              \right.
\end{equation}
Further, Abu-Omar and Kittaneh \cite{AF} improved the inequality (\ref{Ineq1.6}) which states that if $T =[T_{ij }]$
{is} an $n\times n$ operator matrix with $T_{ij }\in \L(\h_j,\h_i)$, then
\begin{equation}\label{Ineq1.7}
  \omega(T)\leq \omega\bra{[t_{ij}]},\,\,\mbox{where}\, t_{ij}=\left\{
                \begin{array}{ll}
                  \omega(T_{ij}), & \hbox{if $i=j$;} \\
                  \norm{T_{ij}}, & \hbox{if $i\neq j$.}
                \end{array}
              \right.
\end{equation}
In \cite{Buz}, Buzano obtained the following extension of the celebrated Cauchy-Schwarz inequality in a real
or complex inner product space $\h$
\begin{equation}\label{Buz-A1}
  \abs{\seq{x,z}\seq{z,y}}\leq \frac{\norm{z}^2}{2}\bra{\norm{x}\norm{y}+\abs{\seq{x,y}}}\,\,\bra{x,y,z\in\h}.
\end{equation}
When $x=y$ this inequality becomes the Cauchy-Schwarz inequality
\begin{equation}\label{CS1}
  \abs{\seq{x,y}}^2\leq \norm{x}^2\norm{y}^2.
\end{equation}
Recently, Khosravi et al. \cite{KDM} proved that if $x,y,z\in\h$  and $\alpha\in\c\setminus\{0\}$, then
\begin{equation}\label{Gbuz}
  \abs{\seq{x,z}\seq{z,y}}\leq \frac{\norm{z}^2}{\abs{\alpha}}\bra{\max\set{1,\abs{\alpha-1}}\norm{x}\norm{y}+\abs{\seq{x,y}}}.
\end{equation}
\begin{remark} When $\alpha=2$ the inequality (\ref{Gbuz}) becomes the Buzano inequality (\ref{Buz-A1}).
\end{remark}
\section{Preliminaries}
In this section, we establish new upper and lower bounds for the numerical radius of $2 \times 2$ operator matrices, improving upon existing results in the literature. Our approach builds on several key lemmas that we present first.

\subsection{Preliminary Lemmas}

We begin with a result about the numerical radius of non-negative matrices:

\begin{lemma} \cite[p. 44]{Horn}\label{lemma-Horn}
Let $A=[a_{ij}]\in M_n(\mathbb{C})$ be such that $a_{ij}\geq 0$ for all $i, j = 1,2,\ldots,n$. Then
\begin{equation*}
  \omega(A)\leq \frac{1}{2}r([a_{ij}+a_{ji}]).
\end{equation*}
\end{lemma}

The following lemma is the well-known Hölder-McCarthy inequality:

\begin{lemma}\cite{kit1}[H\"older-McCarthy inequality]\label{Holder}
Let $A\in\mathcal{B}(\mathcal{H})$, $A\geq 0$ and let $x\in\mathcal{H}$ be any unit vector. Then:
\begin{enumerate}
    \item[(i)] $\langle Ax,x\rangle^r \leq \langle A^rx,x\rangle$ for $r\geq 1$.
    \item[(ii)] $\langle A^rx,x\rangle \leq \langle Ax,x\rangle^r$ for $0<r\leq 1$.
\end{enumerate}
\end{lemma}

Next, we state a version of Jensen's inequality for power functions:

\begin{lemma}\label{J}
Let $a,b>0$ and $0\leq \alpha\leq 1$. Then
\begin{equation}
    a^{\alpha}b^{1-\alpha} \leq \alpha a + (1-\alpha)b \leq \left(\alpha a^r + (1-\alpha)b^r\right)^{\frac{1}{r}} \quad \text{for} \quad r\geq 1.
\end{equation}
\end{lemma}

The following result concerns convex functions of positive operators:

\begin{lemma}\cite{Silva}\label{lemma2.2}
Let $\psi$ be a non-negative convex function on $[0,\infty)$ and $T,S\in\mathcal{B}(\mathcal{H})$ be positive operators. Then
\begin{equation*}
    \left\|\psi\left(\frac{T+S}{2}\right)\right\| \leq \left\|\frac{\psi(T)+\psi(S)}{2}\right\|.
\end{equation*}
In particular,
\begin{equation*}
    \|(T+S)^r\| \leq 2^{r-1}\|T^r+S^r\| \quad \text{for all } r\geq 1.
\end{equation*}
\end{lemma}

We also need Bohr's inequality (see \cite{Bha, Bohr, AA}):

\begin{lemma}[Bohr's inequality]\label{Logain1}
Let $a_i$, $i=1,\ldots,n$ be positive real numbers. Then
\begin{equation}\label{hoopy1}
    \left(\sum_{i=1}^n a_i\right)^r \leq n^{r-1}\sum_{i=1}^n a_i^r \quad \text{for } r\geq 1.
\end{equation}
\end{lemma}

The generalized mixed Schwartz inequality will be particularly useful:

\begin{lemma}\label{Lem2.3}
Let $T\in\mathcal{B}(\mathcal{H})$ and $x,y\in\mathcal{H}$ be any vectors.
\begin{enumerate}
    \item[(i)] If $\alpha,\beta\geq 0$ with $\alpha+\beta=1$, then $|\langle Tx,y\rangle|^2 \leq \langle |T|^{2\alpha}x,x\rangle \langle |T|^{2\beta}y,y\rangle$.
    \item[(ii)] If $f,g$ are non-negative continuous functions on $[0,\infty)$ satisfying $f(t)g(t)=t$ ($t\geq 0$), then $|\langle Tx,y\rangle| \leq \|f(|T|)x\| \|g(|T^*|)y\|$.
\end{enumerate}
\end{lemma}

The next lemma provides fundamental properties of numerical radii for operator matrices:

\begin{lemma}\cite{HK}\label{lemma5.1}
Let $T,S\in\mathcal{B}(\mathcal{H})$. Then:
\begin{enumerate}
    \item[(a)] $\omega\left(\begin{bmatrix} T & 0 \\ 0 & S \end{bmatrix}\right) = \max\{\omega(T),\omega(S)\}$.
    \item[(b)] $\omega\left(\begin{bmatrix} 0 & T \\ S & 0 \end{bmatrix}\right) = \omega\left(\begin{bmatrix} 0 & S \\ T & 0 \end{bmatrix}\right)$.
    \item[(c)] $\omega\left(\begin{bmatrix} T & S \\ S & T \end{bmatrix}\right) = \max\{\omega(T-S),\omega(T+S)\}$.
    In particular, $\omega\left(\begin{bmatrix} 0 & S \\ S & 0 \end{bmatrix}\right) = \omega(S)$.
\end{enumerate}
\end{lemma}

Finally, we state a result about numerical radii of operator matrices:

\begin{lemma} \cite[Theorem 2]{AF}\label{theorem2.7}
Let $S=[S_{ij}]$ be an $n\times n$ operator matrix with $S_{ij}\in \mathcal{L}(\mathcal{H}_j,\mathcal{H}_i)$, $1\leq i,j\leq n$. Then
\begin{equation*}
    \omega(S) \leq \omega([s_{ij}]), \quad \text{where } s_{ij} = \begin{cases}
        \omega(S_{ij}), & \text{if } i=j, \\
        \omega\left(\begin{bmatrix} O & S_{ij} \\ S_{ji} & O \end{bmatrix}\right), & \text{if } i\neq j.
    \end{cases}
\end{equation*}
\end{lemma}
\section{Results}
The upcoming outcome pertains to the numerical radius inequality of an $n\times n$ operator matrix featuring a solitary non-zero row.
\begin{theorem}\label{Theorem2.9} Let $S_j\in \L(\h_j,\h_1)$, ($1\leq j\leq n$). Then
\begin{equation*}
  \omega\bra{\begin{bmatrix}
               S_1 & S_2 &\cdots & S_n \\
               0 & \cdots & \cdots & 0 \\
               \vdots & \cdots & \cdots & \vdots \\
               0 & \cdots & \cdots & 0 \\
             \end{bmatrix}}\leq \frac{1}{2}\bra{\omega(S_1)+\bra{\omega(S_1)+\sum_{j=2}^{n}\norm{S_j}^2}^{\frac{1}{2}}}.
\end{equation*}
\end{theorem}
\begin{proof} By  Lemma \ref{theorem2.7} and  the inequality (\ref{IneqB}), we get
\begin{eqnarray*}
 && \omega \bra{\begin{bmatrix}
               S_1 & S_2 &\cdots & S_n \\
               0 & \cdots & \cdots & 0 \\
               \vdots & \cdots & \cdots & \vdots \\
               0 & \cdots & \cdots & 0 \\
             \end{bmatrix}}\leq \omega\bra{\begin{bmatrix}
               \omega(S_1) & \omega\bra{\begin{bmatrix} 0 &S_{2} \\0& 0 \\\end{bmatrix}} &\cdots & \omega\bra{\begin{bmatrix} 0 &S_{n} \\0 & 0
               \\\end{bmatrix}} \\
               \omega\bra{\begin{bmatrix} 0 &0 \\S_{2}& 0 \\\end{bmatrix}} & 0 & \cdots & 0 \\
               \vdots & \cdots & \cdots & \vdots \\
               \omega\bra{\begin{bmatrix} 0 &0 \\S_{n}& 0 \\\end{bmatrix}} &0& \cdots & 0 \\
             \end{bmatrix}} \\
 &&=\omega\bra{\begin{bmatrix}
               \omega(S_1) & \frac{1}{2}\norm{S_2} &\cdots & \frac{1}{2}\norm{S_n} \\
               \frac{1}{2}\norm{S_2} & 0 & \cdots & 0 \\
               \vdots & \cdots & \cdots & \vdots \\
               \frac{1}{2}\norm{S_n} & 0 & \cdots & 0 \\
             \end{bmatrix}}\\
 &&=\frac{1}{2}r\bra{\begin{bmatrix}
               2\omega(S_1) & \norm{S_2} &\cdots & \norm{S_n} \\
               \norm{S_2} & 0 & \cdots & 0 \\
               \vdots & \cdots & \cdots & \vdots \\
               \norm{S_n} & 0& \cdots & 0 \\
             \end{bmatrix}}=\frac{1}{2}\bra{\omega(S_1)+\bra{\omega(S_1)+\sum_{j=2}^{n}\norm{S_j}^2}^{\frac{1}{2}}}.
\end{eqnarray*}
\end{proof}
\begin{theorem}\label{Theorem2.10} Let $\S=[S_{ij }]$ be an $n\times n$ operator matrix, where $S_{ij}\in\L(\h_j\oplus\h_i)$,
$1 \leq i, j \leq n$, and let $f$ and $g$ be as in Lemma \ref{Lem2.3}. Then
\begin{equation}\label{Ineq2.10}
  \omega^{s}(\S)\leq n^{2s-2} \omega\bra{[s_{ij}]},\,\,\mbox{where}\,\,s_{ij}=\left\{
                                                                   \begin{array}{ll}
                                                                     \frac{1}{2}\omega\bra{f^{2s}(|S_{ii}|)+g^{2s}\bra{|S_{ii}^*|}}, & \hbox{if $i=j$;}
                                                                     \\
                                                                     \norm{S_{ij}}^{s}, & \hbox{if $i\neq j$.}
                                                                   \end{array}
                                                                 \right.
\end{equation}
for all $s\geq 1$.
\end{theorem}
\begin{proof} Let $\x=\begin{bmatrix}
                        x_1 \\
                        x_2\\
                        \vdots \\
                        x_n \\
                      \end{bmatrix}\in\h=\bigoplus_{i=1}^{n}\h_i$ be any unit vector. Then
\begin{eqnarray*}
  &&\abs{\seq{\S\x,\x}}^{s}= \abs{\sum_{i,j=1}^{n}\seq{S_{ij}x_j,x_i}}^{s}\leq n^{2s-2} \sum_{i,j=1}^{n}{\abs{\seq{S_{ij}x_j,x_i}}^{s}}
  \bra{\mbox{by Lemma \ref{Logain1}}}\\
   &=& n^{2s-2} \sum_{i=1}^{n}\abs{\seq{S_{ii}x_i,x_i}}^{s} +n^{2s-2} \sum_{\substack{i,j=1\\ i\neq j}}^{n}\abs{\seq{S_{ij}x_j,x_i}}^{s}\\
   &\leq&n^{2s-2}  \sum_{i=1}^{n}\seq{f^2(|S_{ii}|)x_i,x_i}^{s/2}\seq{g^2(|S_{ii}^*|)x_i,x_i}^{s/2}+n^{2s-2} \sum_{\substack{i,j=1\\ i\neq
   j}}^{n}\norm{S_{ij}}^{s}\norm{x_i}^{s}\norm{x_j}^{s}\\
   &&\bra{\mbox{by Lemma\ref{Lem2.3}}}\\
   &\leq&n^{2s-2}  \sum_{i=1}^{n}\seq{f^{2s}(|S_{ii}|)x_i,x_i}^{1/2}\seq{g^{2s}(|S_{ii}^*|)x_i,x_i}^{1/2}
   {\norm{x_i}^{2s-2}}+n^{2s-2} \sum_{\substack{i,j=1\\ i\neq j}}^{n}\norm{S_{ij}}^{s}\norm{x_i}^{s}\norm{x_j}^{s}\\
   &&\bra{\mbox{by Lemma \ref{Holder}}}\\
   &\leq&\frac{n^{2s-2}}{2}\sum_{i=1}^{n}\seq{\bra{f^{2s}(|S_{ii}|)+g^{2s}(|S_{ii}^*|)}x_i,x_i}{\norm{x_i}^{2s-2}}+n^{2s-2}\sum_{\substack{i,j=1\\
   i\neq j}}^{n}\norm{S_{ij}}^{s}\norm{x_i}^{s}\norm{x_j}^{s}\\
   &\leq& \frac{n^{2s-2}}{2}\sum_{i=1}^{n}\omega\bra{f^{2s}(|S_{ii}|)+g^{2s}(|S_{ii}^*|)}\norm{x_i}^{2s}+n^{2s-2}\sum_{\substack{i,j=1\\ i\neq
   j}}^{n}\norm{S_{ij}}^{s}\norm{x_i}^{s}\norm{x_j}^{s}\\
   &=&n^{2s-2} \sum_{i,j=1}^{n}s_{ij}\norm{x_i}^{s}\norm{x_j}^{s}=\seq{s_{ij}z,z},\mbox{where $z=\begin{bmatrix}
                       \norm {x_1}^{s} \\
                        \norm{x_2}^{s}\\
                        \vdots \\
                        \norm{x_n}^{s} \\
                      \end{bmatrix}\in\c^n$}.
\end{eqnarray*}
Since {$\norm{z}\leq 1$}, we obtain that
\begin{equation*}
  \abs{\seq{\S\x,\x}}^{s}\leq \omega\bra{[s_{ij}]}.
\end{equation*}
Taking the supremum over all unit vectors $\x\in \h$, we get
\begin{equation*}
  \omega^{s}(\S)\leq n^{2s-2} \omega\bra{[s_{ij}]}.
\end{equation*}
\end{proof}
In light of Theorem \ref{Theorem2.10}, we obtain
\begin{corollary} If $\S=\begin{bmatrix} S_{11} &S_{12} \\S_{21} & S_{22} \\\end{bmatrix}\in\L(\h_1\oplus\h_2)$, then
\begin{equation*}
  \omega^{s}(\S)\leq 2^{2s-3}\bra{\tilde{S}_{11}+\tilde{S}_{22}+\sqrt{\bra{\tilde{S}_{11}-\tilde{S}_{22}}^2+\bra{\norm{S_{12}}^{s}+\norm{S_{21}}^{s}
  }^2}},
\end{equation*}
where
\begin{equation*}
  \tilde{S}_{ii}=\frac{1}{2}\omega\bra{f^{2s}(|S_{ii}|)+g^{2s}\bra{|S_{ii}^*|}}\,\,\mbox{for $i=1,2$}.
\end{equation*}
\end{corollary}
\begin{proof} By Theorem \ref{Theorem2.10} and Lemma \ref{lemma-Horn}, we have
\begin{eqnarray*}
 &&\omega^{s}(\S)\leq 2^{2r-2}\omega\bra{\begin{bmatrix}\frac{1}{2}\omega\bra{f^{2s}(|S_{11}|)+g^{2s}\bra{|S_{11}^*|}} &\norm{S_{12}}^{s}
 \\\norm{S_{21}}^{s}
 & \frac{1}{2}\omega\bra{f^{2s}(|S_{22}|)+g^{2s}\bra{|S_{22}^*|}}\\\end{bmatrix}}\\
  &&\leq \frac{2^{2s-2}}{2}r\bra{\begin{bmatrix} 2\tilde{S}_{11} &\norm{S_{12}}^{s}+\norm{S_{21}}^{s}\\\norm{S_{12}}^{s}+\norm{S_{21}}^{s} &
  \tilde{S}_{22} \\\end{bmatrix}}\\
  &&=2^{2s-3}\bra{\tilde{S}_{11}+\tilde{S}_{22}+\sqrt{\bra{\tilde{S}_{11}-\tilde{S}_{22}}^2+\bra{\norm{S_{12}}^{s}+\norm{S_{21}}^{s} }^2}}.
\end{eqnarray*}
\end{proof}
For $\alpha\in (0,1)$, letting $f(t)=t^{\alpha}$ and $g(t)=t^{1-\alpha}$ in Theorem \ref{Theorem2.10}, we obtain the following inequality.
\begin{corollary} Let $\S=[S_{ij }]$ be an $n\times n$ operator matrix, where $S_{ij}\in\L(\h_j\oplus\h_i)$,
$1 \leq i, j \leq n$. Then
\begin{equation*}
  \omega^{s}(\S)\leq n^{2s-2} \omega\bra{[s_{ij}]},\,\,\mbox{where}\,\,s_{ij}={ \left\{
                                                                  \begin{array}{ll}
                                                                     \frac{1}{2}\omega\bra{ |S_{ii}|^{2s\alpha}+|S_{ii}^*|^{2s(1-\alpha)}}, & \hbox{if
                                                                     $i=j$;} \\
                                                                     \norm{S_{ij}}^{s}, & \hbox{if $i\neq j$.}
                                                                   \end{array}
                                                                 \right.}
\end{equation*}
for all $s\geq 1$.
\end{corollary}
\begin{theorem}\label{theorem2.15} Let $\T=\begin{bmatrix} 0 &B \\C& 0 \\\end{bmatrix}\in\L(\h_1\oplus\h_2)$, $r\geq 1$, and let $f$ and $g$
be as in Lemma \ref{Lem2.3}. Then
\begin{equation*}
  \omega^r(\T)\leq 2^{r-2}\omega^{\frac{1}{2}}\bra{f^{2r}(|B|)+g^{2r}(|C^*|)}\omega^{\frac{1}{2}}\bra{f^{2r}(|C|)+g^{2r}(|B^*|)}.
\end{equation*}
\end{theorem}
\begin{proof} Let $\x=\begin{bmatrix} x_1 \\ x_2 \\\end{bmatrix}$ be any unit vector in $\h_1\oplus\h_2$, i.e., $\norm{x_1}^2+\norm{x_2}^2=1$.
Then
\begin{eqnarray*}
 &&\abs{\seq{\T\x,\x}}^{r}=\abs{\seq{Bx_2,x_1}+\seq{Cx_1,x_2}}^r\\
 &&\leq \sbra{\abs{\seq{Bx_2,x_1}}+\abs{\seq{Cx_1,x_2}}}^{r} \bra{\mbox{by the triangle inequality}}\\
  && \leq 2^{r-1}\sbra{\abs{\seq{Bx_2,x_1}}^r+\abs{\seq{Cx_1,x_2}}^r}\bra{\mbox{by Lemma \ref{Logain1}}}\\
  &&\leq 2^{r-1}\sbra{\seq{f^2(|B|)x_2,x_2}^{r/2}\seq{g^2(|B^*|)x_1,x_1}^{r/2}
  +\seq{f^2(|C|)x_1,x_1}^{r/2}\seq{g^2(|C^*|)x_2,x_2}^{r/2}}\\
  &&\bra{\mbox{by Lemma \ref{Lem2.3}}}\\
  &&\leq 2^{r-1}\sbra{\seq{f^{2r}(|B|)x_2,x_2}^{1/2}\seq{g^{2r}(|B^*|)x_1,x_1}^{1/2}
  +\seq{f^{2r}(|C|)x_1,x_1}^{1/2}\seq{g^{2r}(|C^*|)x_2,x_2}^{1/2}}\\
  &&\bra{\mbox{by Lemma \ref{Holder}}}
  \end{eqnarray*}
  \begin{eqnarray*}
    &&\leq 2^{r-1}\bra{\seq{f^{2r}(|B|)x_2,x_2}+\seq{g^{2r}(|C^*|)x_2,x_2}}^{1/2}\bra{\seq{g^{2r}(|B^*|)x_1,x_1}+\seq{f^{2r}(|C|)x_1,x_1}}^{1/2}\\
  &&\bra{\mbox{by the Cauchy-Schwarz inequality}}\\
  &&\leq 2^{r-1}\seq{\bra{f^{2r}(|B|)+g^{2r}(|C^*|)}x_2,x_2}^{1/2}\seq{\bra{g^{2r}(|B^*|)+f^{2r}(|C|)}x_1,x_1}^{1/2}\\
  &&\leq 2^{r-1}\omega^{\frac{1}{2}}\bra{f^{2r}(|B|)+g^{2r}(|C^*|)}\omega^{\frac{1}{2}}\bra{g^{2r}(|B^*|)+f^{2r}(|C|)}\norm{x_1}\norm{x_2}\\
  &&\leq
  2^{r-1}\omega^{\frac{1}{2}}\bra{f^{2r}(|B|)+g^{2r}(|C^*|)}\omega^{\frac{1}{2}}\bra{g^{2r}(|B^*|)+f^{2r}(|C|)}\bra{\frac{\norm{x_1}^2+\norm{x_2}^2}{2}}\\
  &&\bra{\mbox{by the arithmetic-geometric mean inequality}}\\
  &&=2^{r-2}\omega^{\frac{1}{2}}\bra{f^{2r}(|B|)+g^{2r}(|C^*|)}\omega^{\frac{1}{2}}\bra{g^{2r}(|B^*|)+f^{2r}(|C|)}.
\end{eqnarray*}
Now, taking the supremum over all unit vectors $\x\in\h_1\oplus\h_2$, we obtain the desired result.
\end{proof}
Special cases encompassed by Theorem \ref{theorem2.15} are as follows
\begin{corollary} Let $\T=\begin{bmatrix} 0 &B \\C& 0 \\\end{bmatrix}\in\L(\h_1\oplus\h_2)$. Then
 \begin{equation*}
   \omega^r(\T)\leq 2^{r-2}\omega^{\frac{1}{2}}\bra{|B|^{2\alpha r}+|C^*|^{2r(1-\alpha)}}
   \omega^{\frac{1}{2}}\bra{|C|^{2r\alpha}+|B^*|^{2r(1-\alpha)}}
 \end{equation*}
 for all $r\geq 1$ and $\alpha\in [0,1]$.
\end{corollary}
\begin{proof} The result follows immediately from Theorem \ref{theorem2.15} for $f(t)=t^{\alpha}$ and
$g(t)=t^{1-\alpha}$.
\end{proof}
\begin{lemma}\label{Mix-al-be} Let $a,b,e\in\h$ with $\norm{e}=1$, and let $\alpha\in\c\setminus\{0\}$ and $\beta\geq 0$.
Then
\begin{eqnarray}\label{Mixed-Buz}
  \abs{\seq{a,e}\seq{e,b}}^2&\leq& \frac{\beta+(\beta+1)\max\set{1,\abs{\alpha-1}^2}}{\abs{\alpha}^2(1+\beta)}\norm{a}^2\norm{b}^2\nonumber \\
  &+&\frac{1+2(\beta+1)\max\set{1,\abs{\alpha-1}}}{\abs{\alpha}^2(1+\beta)}\norm{a}\norm{b}\abs{\seq{a,b}}.
\end{eqnarray}
\end{lemma}
\begin{proof} We have
\begin{eqnarray*}
  \abs{\seq{a,b}}^2 &\leq&\abs{\seq{a,b}}\norm{a}\norm{b} \\
   &\leq&\abs{\seq{a,b}}\norm{a}\norm{b}+\beta\bra{\norm{a}^2\norm{b}^2-\abs{\seq{a,b}}^2}.
\end{eqnarray*}
Consequently,
\begin{equation}\label{AQ1}
  \abs{\seq{a,b}}^2 \leq \frac{\beta}{\beta+1}\norm{a}^2\norm{b}^2+\frac{1}{\beta+1}\abs{\seq{a,b}}\norm{a}\norm{b}.
\end{equation}
Utilizing the inequality (\ref{Gbuz}), we obtain
\begin{eqnarray}\label{AQ2}
  \abs{\seq{a,e}\seq{e,b}}^2 &\leq&\frac{\max\set{1,\abs{\alpha-1}^2}}{\abs{\alpha}^2}\norm{a}^2\norm{b}^2+\frac{1}{\abs{\alpha}^2}\abs{\seq{a,b}}^2\nonumber \\
   &+&\frac{2\max\set{1,\abs{\alpha-1}}}{\abs{\alpha}^2} \abs{\seq{a,b}}\norm{a}\norm{b}.
\end{eqnarray}
Combining the inequalities (\ref{AQ1}) and (\ref{AQ2}), we have
\begin{eqnarray*}
  \abs{\seq{a,e}\seq{e,b}}^2 &\leq&\frac{\beta+(\beta+1)\max\set{1,\abs{\alpha-1}^2}}{\abs{\alpha}^2(1+\beta)}\norm{a}^2\norm{b}^2 \\
   &+& \frac{1+2(\beta+1)\max\set{1,\abs{\alpha-1}}}{\abs{\alpha}^2(1+\beta)}\norm{a}\norm{b}\abs{\seq{a,b}}.
\end{eqnarray*}
\end{proof}
\begin{lemma}\label{Buzano} If $a,b,e\in\h$ with $\norm{e}=1$ and $\beta\geq 0$, then
\begin{equation}\label{Imediae1}
  \abs{\seq{a,e}\seq{e,b}}^2\leq
  \frac{1}{4}\bra{\frac{2\beta+1}{\beta+1}\norm{a}^2\norm{b}^2+\frac{2\beta+3}{\beta+1}\norm{a}\norm{b}\abs{\seq{a,b}}}.
\end{equation}
\end{lemma}
\begin{proof} Letting $\alpha=2$ in Lemma \ref{Mix-al-be}.
\end{proof}
\begin{lemma}\label{Ramadan-Kareem1}Let $a,b,e\in\h$ with $\norm{e}=1$, and let $\alpha\in\c\setminus\{0\}$ and $\beta\geq 0$.
Then
 $$\abs{\seq{a,e}\seq{e,b}}^2\leq \frac{2(\beta+1)\max\set{1,\abs{\alpha-1}^2}+2\beta}{\abs{\alpha}^2(\beta+1)}\norm{a}^2\norm{b}^2+\frac{2}{\abs{\alpha}^2(\beta+1)}\abs{\seq{a,b}}^2.$$
\end{lemma}
\begin{proof} By similar discussion with Lemma \ref{Mix-al-be}, we have
$$\abs{\seq{a,b}}^2\leq \abs{\seq{a,b}}^2+\beta\bra{\norm{a}^2\norm{b}^2-\abs{\seq{a,b}}^2}.$$
This indicates that
\begin{equation}\label{NM1}
  \abs{\seq{a,b}}^2\leq \frac{\beta}{\beta+1}\norm{a}^2\norm{b}^2+\frac{1}{\beta+1}\abs{\seq{a,b}}^2.
\end{equation}
By applying Bohr's inequality on the inequality (\ref{Gbuz}), we have
\begin{equation}\label{NM2}
 \abs{\seq{a,e}\seq{e,b}}^2\leq \frac{2\max\set{1,\abs{\alpha-1}^2}}{\abs{\alpha}^2}\norm{a}^2\norm{b}^2+\frac{2}{\abs{\alpha}^2}\abs{\seq{a,b}}^2.
\end{equation}
Then, according to the inequalities (\ref{NM1}) and (\ref{NM2}), one has
\begin{eqnarray*}
  \abs{\seq{a,e}\seq{e,b}}^2 &\leq& \frac{2\max\set{1,\abs{\alpha-1}^2}}{\abs{\alpha}^2}\norm{a}^2\norm{b}^2+\frac{2}{\abs{\alpha}^2}\abs{\seq{a,b}}^2\\
   &\leq&\frac{2\max\set{1,\abs{\alpha-1}^2}}{\abs{\alpha}^2}\norm{a}^2\norm{b}^2
   +\frac{2}{\abs{\alpha}^2}\bra{\frac{\beta}{\beta+1}\norm{a}^2\norm{b}^2+\frac{1}{\beta+1}\abs{\seq{a,b}}^2}\\
   &=& \frac{2(\beta+1)\max\set{1,\abs{\alpha-1}^2}+2\beta}{\abs{\alpha}^2(\beta+1)}\norm{a}^2\norm{b}^2+\frac{2}{\abs{\alpha}^2(\beta+1)}\abs{\seq{a,b}}^2.
\end{eqnarray*}
\end{proof}
\begin{lemma}If $a,b,e\in\h$ with $\norm{e}=1$ and $\beta\geq 0$, then
  \begin{equation}\label{Ramad13}
    \abs{\seq{a,e}\seq{e,b}}^2\leq\frac{1}{2}\bra{\frac{2\beta+1}{\beta+1}\norm{a}^2\norm{b}^2
    +\frac{1}{\beta+1}\abs{\seq{a,b}}^2}.
  \end{equation}
\end{lemma}
\begin{proof} Letting $\alpha=2$ in Lemma \ref{Ramadan-Kareem1}.
\end{proof}
\begin{theorem}\label{MOBY-A1} Let $B,C\in\bh$, $\alpha\in\c\setminus\{0\}$ and $\beta\geq 0$. Then
\begin{eqnarray}\label{Ram-K1}
  \omega^4\bra{\begin{bmatrix} O& B\\ C& O \\\end{bmatrix}}&\leq&\frac{\delta_1}{4}\max\set{\norm{|C|^2+|B^*|^2}^2,\norm{|B|^2+|C^*|^2}^2}\nonumber\\
  &+&\delta_2\max\set{\omega^2\bra{BC},\omega^2(CB)},
\end{eqnarray}
 where $\delta_1=\frac{2(\beta+1)\max\set{1,\abs{\alpha-1}^2}+2\beta}{\abs{\alpha}^2(\beta+1)}$ and $\delta_2=\frac{2}{\abs{\alpha}^2(\beta+1)}$.
\end{theorem}
\begin{proof}Let $\x$ be any unit vector in $\hh$, $\delta_1=\frac{2(\beta+1)\max\set{1,\abs{\alpha-1}^2}+2\beta}{\abs{\alpha}^2(\beta+1)}$ and $\delta_2=\frac{2}{\abs{\alpha}^2(\beta+1)}$,  and let $\T=\begin{bmatrix} O& B\\ C& O \\\end{bmatrix}$. Then
  \begin{eqnarray*}
    \abs{\seq{\T\x,\x}}^4 &\leq&\delta_1\norm{\T\x}^2\norm{\T^*\x}^2+\delta_2\abs{\seq{\T\x,\T^*\x}}^2 \bra{\mbox{by Lemma \ref{Ramadan-Kareem1}}}\\
     &=&\delta_1\bra{\sqrt{\seq{\abs{\T}^2\x,\x}\seq{\abs{\T^*}^2\x,\x}}}^2+\delta_2\abs{\seq{\T^2\x,\x}}^2\\
     &\leq& \frac{\delta_1}{4}\bra{\seq{\abs{\T}^2\x,\x}+\seq{\abs{\T^*}^2\x,\x}}^2+\delta_2\abs{\seq{\T^2\x,\x}}^2\\
     &&\bra{\mbox{by the arithmetic-geometric mean inequality}}
      \end{eqnarray*}
      \begin{eqnarray*}
     &\leq& \frac{\delta_1}{4}\norm{|\T|^2+|\T^*|^2}^2+\delta_2\omega^2\bra{\T^2}\\
     &\leq& \frac{\delta_1}{4}\max\set{\norm{|C|^2+|B^*|^2}^2,\norm{|B|^2+|C^*|^2}^2}+\delta_2\max\set{\omega^2\bra{BC},\omega^2(CB)}.
  \end{eqnarray*}
  By taking the supremum over all vectors of $\x\in\hh$ and using Lemma \ref{lemma5.1}, we obtain
  $$\omega^4(\T)\leq \frac{\delta_1}{4}\max\set{\norm{|C|^2+|B^*|^2}^2,\norm{|B|^2+|C^*|^2}^2}+\delta_2\max\set{\omega^2\bra{BC},\omega^2(CB)}.$$
\end{proof}
Letting $B=C=M$ in Theorem \ref{MOBY-A1} and using Lemma \ref{lemma5.1}, we have
\begin{corollary}\label{MOBY-A2} Let $M\in\bh$. Then
$$\omega^4(M)\leq \frac{\delta_1}{4}\norm{|M|^2+|M^*|^2}^2+\delta_2\omega^2(M^2).$$
\end{corollary}
\begin{remark} Corollary \ref{MOBY-A2} with $\alpha=2$ is sharper than the second inequality in the inequality (\ref{Ineq1.3}).
As a matter of fact,
\begin{eqnarray*}
  &&\frac{2\beta+1}{8(\beta+1)}\norm{|M|^2+|M^*|^2}^2+\frac{1}{2(\beta+1)}\omega^2(M^2) \\
  &\leq& \frac{2\beta+1}{8(\beta+1)}\norm{|M|^2+|M^*|^2}^2+\frac{1}{2(\beta+1)}\omega^4(M)  \\
   &\leq& \frac{2\beta+1}{8(\beta+1)}\norm{|M|^2+|M^*|^2}^2+\frac{1}{8(\beta+1)}\norm{|M|^2+|M^*|^2}^2\\
   &=& \frac{1}{4}\norm{|M|^2+|M^*|^2}^2.
\end{eqnarray*}
\end{remark}
\begin{theorem}\label{Ramadan1} Let $B,C\in\bh$. Then
\begin{eqnarray*}
  \omega^{4r}\bra{\begin{bmatrix} O& B\\ C& O \\\end{bmatrix}}&\leq& \frac{\gamma_1}{16}\max\set{\norm{ |C|^{2r}+|B^*|^{2r}}^2,\norm{|B|^{2r} +|C^*|^{2r}}^2}\\
   &+&\frac{\gamma_2}{8}\max\set{\norm{ |C|^{2r}+|B^*|^{2r}},\norm{|B|^{2r} +|C^*|^{2r}}}\max\set{\omega^r(BC),\omega^r(CB)}
\end{eqnarray*}
for all $r\geq 1$, where $\gamma_1=\frac{2\beta+1}{\beta+1}$ and $\gamma_2=\frac{2\beta+3}{\beta+1}$.
\end{theorem}
\begin{proof} Let $\x$ be any unit vector in $\hh$, $\gamma_1=\frac{2\beta+1}{\beta+1}, \gamma_2=\frac{2\beta+3}{\beta+1}$ and let $\T=\begin{bmatrix} O& B\\ C& O \\\end{bmatrix}$. Then
\begin{eqnarray*}
  \abs{\seq{\T\x,\x}}^{4} &\leq& \frac{\gamma_1}{4}\norm{\T\x}^2\norm{\T^*\x}^2+
  \frac{\gamma_2}{4}\abs{\seq{\T^2\x,\x}}\norm{\T\x}\norm{\T^*\x}\bra{\mbox{by Lemma \ref{Buzano}}} \\
   &\leq&\bra{\frac{\gamma_1}{4}\norm{\T\x}^{2r}\norm{\T^*\x}^{2r}+
   \frac{\gamma_2}{4}\norm{\T\x}^{r}\norm{\T^*\x}^{r}\abs{\seq{\T^2\x,\x}}^r}^{\frac{1}{r}}\bra{\mbox{by Lemma \ref{J}}}.
   \end{eqnarray*}
Consequently,
\begin{eqnarray*}
   \abs{\seq{\T\x,\x}}^{4r} &\leq& \frac{\gamma_1}{4}\norm{\T\x}^{2r}\norm{\T^*\x}^{2r}+
   \frac{\gamma_2}{4}\norm{\T\x}^{r}\norm{\T^*\x}^{r}\abs{\seq{\T^2\x,\x}}^r\\
  &=& \frac{\gamma_1}{4}\bra{\sqrt{\seq{\T^*\T\x,\x}^{r}\seq{\T\T^*\x,\x}^{r}}}^2+
   \frac{\gamma_2}{4}\abs{\seq{\T^2\x,\x}}^r\sqrt{\seq{\T^*\T\x,\x}^{r}\seq{\T\T^*\x,\x}^{r}}
    \end{eqnarray*}
      \begin{eqnarray*}
  &\leq& \frac{\gamma_1}{4}\bra{\sqrt{\seq{\begin{bmatrix} |C|^2& O\\ O& |B|^2 \\\end{bmatrix}\x,\x}^{r}\seq{\begin{bmatrix} |B^*|^2& O\\ O& |C^*|^2
  \\\end{bmatrix}\x,\x}^{r}}}^2\\
  &+&\frac{\gamma_2}{4}\abs{\seq{\T^2\x,\x}}^r\sqrt{\seq{\begin{bmatrix} |C|^2& O\\ O& |B|^2 \\\end{bmatrix}\x,\x}^{r}\seq{\begin{bmatrix} |B^*|^2& O\\ O& |C^*|^2
  \\\end{bmatrix}\x,\x}^{r}}\\
  &\leq& \frac{\gamma_1}{16}\bra{\seq{\begin{bmatrix} |C|^2& O\\ O& |B|^2 \\\end{bmatrix}\x,\x}^{r}+\seq{\begin{bmatrix} |B^*|^2& O\\ O& |C^*|^2
  \\\end{bmatrix}\x,\x}^{r}}^2\\
  &+&\frac{\gamma_2}{8}\bra{\seq{\begin{bmatrix} |C|^2& O\\ O& |B|^2 \\\end{bmatrix}\x,\x}^{r}+\seq{\begin{bmatrix} |B^*|^2& O\\ O& |C^*|^2
  \\\end{bmatrix}\x,\x}^{r}}\abs{\seq{\T^2\x,\x}}^r\\
  &&\bra{\mbox{by the arithmetic-geometric mean inequality}}\\
  &\leq& \frac{\gamma_1}{16}\bra{\seq{\begin{bmatrix} |C|^{2r}& O\\ O& |B|^{2r} \\\end{bmatrix}\x,\x}
  +\seq{\begin{bmatrix} |B^*|^{2r}& O\\ O& |C^*|^{2r} \\\end{bmatrix}\x,\x}}^2\\
  &+&\frac{\gamma_2}{8}\bra{\seq{\begin{bmatrix} |C|^{2r}& O\\ O& |B|^{2r} \\\end{bmatrix}\x,\x}+\seq{\begin{bmatrix} |B^*|^{2r}& O\\ O& |C^*|^{2r}
  \\\end{bmatrix}\x,\x}}\abs{\seq{\T^2\x,\x}}^r\,\,\bra{\mbox{by Lemma \ref{Holder}}}\\
   &\leq& \frac{\gamma_1}{16}\norm{\begin{bmatrix} |C|^{2r}+|B^*|^{2r}& O\\ O& |B|^{2r} +|C^*|^{2r} \\\end{bmatrix}}^2\\
   &+&\frac{\gamma_2}{8}\norm{\begin{bmatrix} |C|^{2r}+|B^*|^{2r}& O\\ O& |B|^{2r} +|C^*|^{2r} \\\end{bmatrix}}\abs{\seq{\begin{bmatrix} BC&O\\ O& CB \\\end{bmatrix}\x,\x}}^{r}\\
   &\leq&\frac{\gamma_1}{16}\max\set{\norm{ |C|^{2r}+|B^*|^{2r}}^2,\norm{|B|^{2r} +|C^*|^{2r}}^2}\\
   &+&\frac{\gamma_2}{8}\max\set{\norm{ |C|^{2r}+|B^*|^{2r}},\norm{|B|^{2r} +|C^*|^{2r}}}\max\set{\omega^r(BC),\omega^r(CB)}.
\end{eqnarray*}
By taking the supremum over all vectors of $\x\in\hh$ and using Lemma \ref{lemma5.1}, we obtain
\begin{eqnarray*}
  \omega^{4r}\bra{\begin{bmatrix} O& B\\ C& O \\\end{bmatrix}}&\leq& \frac{\gamma_1}{16}\max\set{\norm{ |C|^{2r}+|B^*|^{2r}}^2,\norm{|B|^{2r} +|C^*|^{2r}}^2}\\
   &+&\frac{\gamma_2}{8}\max\set{\norm{ |C|^{2r}+|B^*|^{2r}},\norm{|B|^{2r} +|C^*|^{2r}}}\max\set{\omega^r(BC),\omega^r(CB)}.
\end{eqnarray*}
\end{proof}
Letting $B=C=M$ in Theorem \ref{Ramadan1} and using Lemma \ref{lemma5.1}, we have
\begin{corollary} Let $M\in\bh$. Then
$$\omega^{4r}(M)\leq \frac{1}{16}\bra{\frac{2\beta+1}{\beta+1}}\norm{|M|^{2r}+|M^*|^{2r}}^2
+\frac{1}{8}\bra{\frac{2\beta+3}{\beta+1}}\norm{|M|^{2r}+|M^*|^{2r}}\omega^{r}(M^2) $$
  for all $r\geq 1$.
\end{corollary}
Based on Lemma \ref{Buzano} and the convexity of the function $f(t)=t^r\,(r\geq 1)$, we have
\begin{lemma} \label{modified-Buzano}  Let $a,b,e\in\h$ with $\norm{e}=1$ and $\beta\geq 0$ . Then
  \begin{equation}\label{Mohd-Ineq.}
  \abs{\seq{a,e}\seq{e,b}}^{2r}\leq \frac{1}{4}\bra{\bra{\frac{2\beta+1}{\beta+1}}\norm{a}^{2r}\norm{b}^{2r}+
  \bra{\frac{2\beta+3}{\beta+1}}\norm{a}^r\norm{b}^r\abs{\seq{a,b}}^{r}}
\end{equation}
for all $r\geq 1$.
\end{lemma}
Now, we are in position to applied Lemma \ref{modified-Buzano} to
establish a new upper bound for the numerical radii of $2\times 2$ operator matrices.
The following outcome is stated as:
\begin{theorem}\label{Theorem2.16} Let $B,C\in\bh$ and let $f, g$ be nonnegative continuous functions on $[0,\infty)$
satisfying $f(t)g(t)=t$, ($t\geq0$). Then
\begin{eqnarray}\label{MALIK-A1}
  \omega^{4r}\bra{\begin{bmatrix} O& B\\ C& O \\\end{bmatrix}}&\leq&\frac{\gamma_1}{4}\max\set{\norm{ |C|^{4r}+|B^*|^{4r}},\norm{|B|^{4r} +|C^*|^{4r}}}\nonumber \\
   &+&\frac{\gamma_2}{2}\max\set{\norm{ |C|^{2r}+|B^*|^{2r}},\norm{|B|^{2r} +|C^*|^{2r}}}\times\\
   &&\max\set{\norm{\frac{1}{p}f^{pr}(|BC|)+\frac{1}{q}g^{qr}(|C^*B^*|)},\norm{\frac{1}{p}f^{pr}(|CB|)+\frac{1}{q}g^{qr}(|B^*C^*|)}}\nonumber
\end{eqnarray}
for all $r \geq1$, $p, q > 1$ with $\frac{1}{p}+\frac{1}{q}=1$ and $pr,qr\geq  2$, where $\gamma_1=\frac{2\beta+1}{4(\beta+1)}$ and $\gamma_2=\frac{2\beta+3}{4(\beta+1)}$.
\end{theorem}
\begin{proof} Let $\x$ be any unit vector in $\hh$, $\gamma_1=\frac{2\beta+1}{4(\beta+1)}$, $\gamma_2=\frac{2\beta+3}{4(\beta+1)}$, and let $\T=\begin{bmatrix} O& B\\ C& O \\\end{bmatrix}$. Then, we have
\begin{eqnarray}\label{Isra1}
  \abs{\seq{\T^2\x,\x}}^{r} &\leq&\norm{f(|\T^2|)}^{r}\norm{g(|(\T^{2})^*|)}^r\bra{\mbox{by Lemma \ref{Lem2.3}}} \nonumber\\
  &=& \seq{f^2\bra{|\T^2|}\x,\x}^{\frac{r}{2}}\seq{g^2(|(\T^{2})^*|)\x,\x}^{\frac{r}{2}}\nonumber\\
  &\leq& \frac{1}{p}\seq{f^2\bra{|\T^2|}\x,\x}^{\frac{pr}{2}}+\frac{1}{q}\seq{g^2(|(\T^{2})^*|)\x,\x}^{\frac{qr}{2}}
  \bra{\mbox{by Young's inequality}}\nonumber\\
  &\leq&\frac{1}{p}\seq{f^{pr}\bra{|\T^2|}\x,\x}+\frac{1}{q}\seq{g^{qr}(|(\T^{2})^*|)\x,\x}\bra{\mbox{by Lemma \ref{Holder}}}\\
  &=&\frac{1}{p}\seq{f^{pr}\bra{\begin{bmatrix} |BC|& O\\ O& |CB| \\\end{bmatrix}}\x,\x}+\frac{1}{q}
  \seq{g^{qr}\bra{\begin{bmatrix} |C^*B^*|& O\\ O& |B^*C^*| \\\end{bmatrix}}\x,\x}\nonumber\\
  &=&\seq{\begin{bmatrix}\frac{1}{p}f^{pr}(|BC|)+\frac{1}{q}g^{qr}(|C^*B^*|)& O\\ O&
  \frac{1}{p}f^{pr}(|CB|)+\frac{1}{q}g^{qr}(|B^*C^*|)  \\\end{bmatrix}\x,\x}\nonumber\\
  &\leq& \max\set{\norm{\frac{1}{p}f^{pr}(|BC|)+\frac{1}{q}g^{qr}(|C^*B^*|)},\norm{\frac{1}{p}f^{pr}(|CB|)+\frac{1}{q}g^{qr}(|B^*C^*|)}}.\nonumber
\end{eqnarray}
Now, by using Lemma \ref{modified-Buzano}, we get
\begin{eqnarray*}
  \abs{\seq{\T\x,\x}}^{4r} &\leq&\gamma_1\norm{\T\x}^{2r}\norm{\T^*\x}^{2r}+\gamma_2\norm{\T\x}^{r}\norm{\T^*\x}^{r}\abs{\seq{\T\x,\T^*\x}}^r \\
   &\leq& \gamma_1\seq{\T^*\T\x,\x}^{r}\seq{\T^*\T\x,\x}^{r}+\gamma_2\seq{\T^*\T\x,\x}^{r/2}\seq{\T^*\T\x,\x}^{r/2} \abs{\seq{\T^2\x,\x}}^r\\
   &\leq& \frac{\gamma_1}{4}\bra{\seq{\T^*\T\x,\x}^{2r}+\seq{\T^*\T\x,\x}^{2r}}+\frac{\gamma_2}{2}\bra{\seq{\T^*\T\x,\x}^{r}+\seq{\T^*\T\x,\x}^{r}}
   \abs{\seq{\T^2\x,\x}}^r\\
   &&\bra{\mbox{by the arithmetic-geometric mean inequality}}\\
   &\leq& \frac{\gamma_1}{4}\bra{\seq{\bra{\abs{\T}^{4r}+\abs{\T^*}^{4r}}\x,\x}}+\frac{\gamma_2}{2}
   \bra{\seq{\bra{\abs{\T}^{2r}+\abs{\T^*}^{2r}}\x,\x}}
   \abs{\seq{\T^2\x,\x}}^r\\
   &&\bra{\mbox{by the H\"older-McCarthy inequality}}\\
   &\leq&\frac{\gamma_1}{4}\norm{\abs{\T}^{4r}+\abs{\T^*}^{4r}}+\frac{\gamma_2}{2}\norm{\abs{\T}^{2r}+\abs{\T^*}^{2r}}\abs{\seq{\T^2\x,\x}}^r
\end{eqnarray*}
By taking the supremum over all vectors of $\x\in\hh$ and using Lemma \ref{lemma5.1} and the inequality (\ref{Isra1}), we obtain
\begin{eqnarray*}
  \omega^{4r}(\T) &\leq&\frac{\gamma_1}{4}\max\set{\norm{ |C|^{4r}+|B^*|^{4r}},\norm{|B|^{4r} +|C^*|^{4r}}} \\
   &+&\frac{\gamma_2}{2}\max\set{\norm{ |C|^{2r}+|B^*|^{2r}},\norm{|B|^{2r} +|C^*|^{2r}}}\times\\
   &&\max\set{\norm{\frac{1}{p}f^{pr}(|BC|)+\frac{1}{q}g^{qr}(|C^*B^*|)},\norm{\frac{1}{p}f^{pr}(|CB|)+\frac{1}{q}g^{qr}(|B^*C^*|)}}.
\end{eqnarray*}
This completes the proof.
\end{proof}
The inequality (\ref{MALIK-A1}) leads to various numerical radius inequalities when considered as specific instances. As an illustration, when we
choose $f(t))=t^{\lambda}$ and $g(t)=t^{1-\lambda}$ and $B=C=M$ and set $p=q=2$ in (\ref{MALIK-A1}), the ensuing outcome by using Lemma \ref{lemma5.1} is
as follows.
\begin{corollary}\label{Cor3.18} Let $M\in\bh$. Then
\begin{equation}\label{Ineq2.17}
  \omega^{4r}(M)\leq \frac{2\beta+1}{16(\beta+1)}\norm{ |M|^{4r}+|M^*|^{4r}}+\frac{2\beta+3}{16(\beta+1)}\norm{ |M|^{2r}+|M^*|^{2r}}
  \norm{|M^2|^{2r\lambda}+|M^{*2}|^{2r(1-\lambda)}}
\end{equation}
for every $r\geq 1$ and $\lambda\in [0,1]$. In particular, if $\lambda=\frac{1}{2}$, then
\begin{equation*}
  \omega^{4r}(M)\leq \frac{2\beta+1}{16(\beta+1)}\norm{ |M|^{4r}+|M^*|^{4r}}+\frac{2\beta+3}{16(\beta+1)}\norm{ |M|^{2r}+|M^*|^{2r}}
  \norm{|M^2|^{r}+|M^{*2}|^{r}}.
\end{equation*}
\end{corollary}
\begin{theorem}\label{Theorem 2.16} Let $B,C\in\bh$, $\alpha\in\c\setminus\{0\}$ and $\beta\geq 0$. Then
\begin{eqnarray}
 \omega^4\bra{\begin{bmatrix} O& B\\ C& O \\\end{bmatrix}} &\leq&\frac{\chi_1}{4}\max\set{\norm{|C|^{2\mu}+|B^*|^{2(1-\mu)}}^2,\norm{|B|^{2\mu}+|C^*|^{2(1-\mu)} }^2}\nonumber \\
     &+& \frac{\chi_2}{2}\max\set{\norm{|C|^{2\mu}+|B^*|^{2(1-\mu)}},\norm{|B|^{2\mu}+|C^*|^{2(1-\mu)} }}\times\nonumber\\
     &&\max\set{\omega^2\bra{|B^*|^{2(1-\mu)}|C|^{2\mu}},\omega^2\bra{|C^*|^{2(1-\mu)}|B|^{2\mu}}}
\end{eqnarray}
for all $\mu\in[0,1]$, where $\chi_1=\frac{\beta+(\beta+1)\max\set{1,\abs{\alpha-1}^2}}{\abs{\alpha}^2(\beta+1)}$ and
$\chi_2=\frac{1+2(\beta+1)\max\set{1,\abs{\alpha-1}}}{\abs{\alpha}^2(\beta+1)}$ .
\end{theorem}
\begin{proof} Let $\x\in\hh$ be any unit vector, $\chi_1=\frac{\beta+(\beta+1)\max\set{1,\abs{\alpha-1}^2}}{\abs{\alpha}^2(\beta+1)}$,
$\chi_2=\frac{1+2(\beta+1)\max\set{1,\abs{\alpha-1}}}{\abs{\alpha}^2(\beta+1)}$, and let $\T=\begin{bmatrix} O& B\\ C& O \\\end{bmatrix}$. Then
\begin{eqnarray*}
  \abs{\seq{\T\x,\x}}^{4}&\leq& \seq{|\T|^{2\mu}\x,\x}^{2}\seq{|\T^*|^{2(1-\mu)}\x,\x}^{2}\,\bra{\mbox{by Lemma \ref{Lem2.3} and Lemma \ref{Holder}}} \\
   &=& \seq{|\T|^{2\mu}\x,\x}^{2}\seq{\x,|\T^*|^{2(1-\mu)}\x}^{2}\\
   &\leq& \chi_1\bra{\sqrt{\norm{|\T|^{\mu}\x}^2\norm{|\T^*|^{(1-\mu)}\x}^2}}^2+\chi_2\bra{\sqrt{\norm{|\T|^{\mu}\x}\norm{|\T^*|^{(1-\mu)}\x}}}^2
   \abs{\seq{|\T|^{\mu}\x,|\T^*|^{(1-\mu)}\x}}^2\\
   &&\bra{\mbox{by Lemma \ref{Mix-al-be}}}\\
   &\leq& \frac{\chi_1}{4}\bra{\norm{|\T|^{\mu}\x}^2+\norm{|\T^*|^{(1-\mu)}\x}^2}^2+
   \frac{\chi_2}{2}\bra{\norm{|\T|^{\mu}\x}^2+\norm{|\T^*|^{(1-\mu)}\x}^2}\abs{\seq{|\T^*|^{(1-\mu)}|\T|^{\mu}\x,\x}}^2\\
   &&\bra{\mbox{by the arithmetic-geometric mean inequality}}\\
   &\leq& \frac{\chi_1}{4}\bra{\seq{\bra{|\T|^{2\mu}+|\T^*|^{2(1-\mu)}}\x,x}}^2+\frac{\chi_2}{2}\bra{\seq{\bra{|\T|^{2\mu}+|\T^*|^{2(1-\mu)}}\x,x}}
   \abs{\seq{|\T^*|^{(1-\mu)}|\T|^{\mu}\x,\x}}^2
   \end{eqnarray*}
   \begin{eqnarray*}
   &\leq& \frac{\chi_1}{4}\norm{|\T|^{2\mu}+|\T^*|^{2(1-\mu)}}^2+\frac{\chi_2}{2}\norm{|\T|^{2\mu}+|\T^*|^{2(1-\mu)}}\omega^2\bra{|\T^*|^{(1-\mu)}|\T|^{\mu}}\\
   &=&\frac{\chi_1}{4}\norm{\begin{bmatrix} |C|^{2\mu}+|B^*|^{2(1-\mu)}& O\\ O&|B|^{2\mu}+|C^*|^{2(1-\mu)} \\\end{bmatrix}}^2\\
   &+& \frac{\chi_2}{2}\norm{\begin{bmatrix} |C|^{2\mu}+|B^*|^{2(1-\mu)}& O\\ O&|B|^{2\mu}+|C^*|^{2(1-\mu)} \\\end{bmatrix}}
   \omega^{2}\bra{\begin{bmatrix} |B^*|^{2(1-\mu)}|C|^{2\mu}& O\\ O& |C^*|^{2(1-\mu)}|B|^{2\mu} \\\end{bmatrix}}
   \end{eqnarray*}
  By taking the supremum over all vectors of $\x\in\hh$ and using Lemma \ref{lemma5.1}, we obtain
  \begin{eqnarray*}
    \omega^4(\T) &\leq&\frac{\chi_1}{4}\max\set{\norm{|C|^{2\mu}+|B^*|^{2(1-\mu)}}^2,\norm{|B|^{2\mu}+|C^*|^{2(1-\mu)} }^2} \\
     &+& \frac{\chi_2}{2}\max\set{\norm{|C|^{2\mu}+|B^*|^{2(1-\mu)}},\norm{|B|^{2\mu}+|C^*|^{2(1-\mu)} }}\times\\
     &&\max\set{\omega^2\bra{|B^*|^{2(1-\mu)}|C|^{2\mu}},\omega^2\bra{|C^*|^{2(1-\mu)}|B|^{2\mu}}}.
  \end{eqnarray*}
\end{proof}
Letting $B=C=M$ in Theorem \ref{Theorem 2.16} and using Lemma \ref{lemma5.1}, we have
\begin{corollary} Let $M\in\bh$, $\alpha\in\c\setminus\{0\}$ and $\beta\geq 0$. Then
$$ \omega^4(M)=\frac{\chi_1}{4}\norm{|M|^{2\mu}+|M^*|^{2(1-\mu)}}^2+\frac{\chi_2}{2}\norm{|M|^{2\mu}
+|M^*|^{2(1-\mu)}}\omega^2\bra{|M^*|^{2(1-\mu)}|M|^{2\mu}},$$
where $\chi_1=\frac{\beta+(\beta+1)\max\set{1,\abs{\alpha-1}^2}}{\abs{\alpha}^2(\beta+1)}$ and
$\chi_2=\frac{1+2(\beta+1)\max\set{1,\abs{\alpha-1}}}{\abs{\alpha}^2(\beta+1)}$ .
In particular, if $\alpha=2$, then
$$\omega^4(M)\leq \frac{2\beta+1}{16(\beta+1)}\norm{|M|^{2\mu}+|M^*|^{2(1-\mu)}}^2
+\frac{2\beta+3}{8(\beta+1)}\norm{|M|^{2\mu}+|M^*|^{2(1-\mu)}}\omega^2\bra{|M^*|^{2(1-\mu)}|M|^{2\mu}}.$$
\end{corollary}
The following Lemma is very useful in the sequel and it can be found in \cite[p. 148]{Drag5}.
\begin{lemma}\label{Drag} Let $a,b,e\in\h$ with $\norm{e}=1$. Then
$$\abs{\seq{a,e}}^2+\abs{\seq{e,b}}^2\leq \sqrt{\seq{a,a}^2+\seq{b,b}^2+2\abs{\seq{a,b}}^2}.$$
\end{lemma}
\begin{theorem}\label{KZ} Let $T,S,X,Y\in\bh$, $\alpha\in\c\setminus\{0\}$ and $\beta\geq 0$. Then
\begin{eqnarray}
  \omega^4\bra{\begin{bmatrix} T& X\\ Y& S \\\end{bmatrix}}  &\leq& \bra{2+4\chi_1}\max\set{\norm{|T|^4+|X^*|^4},\norm{|S|^4+|Y^*|^4}}
  +2\omega^2\bra{\begin{bmatrix} O& XS\\ YT& O \\\end{bmatrix}}\nonumber\\
   &+&4\chi_2\max\set{\norm{|T|^2+|X^*|^2},\norm{|S|^2+|Y^*|^2}}w\bra{\begin{bmatrix} O& XS\\ YT& O \\\end{bmatrix}},
\end{eqnarray}
 where $\chi_1=\frac{\beta+(\beta+1)\max\set{1,\abs{\alpha-1}^2}}{\abs{\alpha}^2(\beta+1)}$ and
$\chi_2=\frac{1+2(\beta+1)\max\set{1,\abs{\alpha-1}}}{\abs{\alpha}^2(\beta+1)}$.
\end{theorem}
\begin{proof} Let  $\x\in\hh$ be any unit vector, $\chi_1=\frac{\beta+(\beta+1)\max\set{1,\abs{\alpha-1}^2}}{\abs{\alpha}^2(\beta+1)}$,
$\chi_2=\frac{1+2(\beta+1)\max\set{1,\abs{\alpha-1}}}{\abs{\alpha}^2(\beta+1)}$, and let $\T=\begin{bmatrix} T& X\\ Y& S \\\end{bmatrix}$,
$M=\begin{bmatrix} O& X\\ Y& O\\\end{bmatrix}$, $P=\begin{bmatrix} T& O\\ O& S \\\end{bmatrix}$ and $R=\begin{bmatrix} O& XS\\ YT& O \\\end{bmatrix}$.
Then $MP=R$,
$$|P|^4+|M^*|^4=\begin{bmatrix} |T|^4+|X^*|^4& O\\ O& |S|^4+|Y^*|^4 \\\end{bmatrix}\,\,\mbox{and}\,\,|P|^2+|M^*|^2
=\begin{bmatrix} |T|^2+|X^*|^2& O\\ O& |S|^2+|Y^*|^2 \\\end{bmatrix} .$$
Therefore,
\begin{eqnarray*}
  \abs{\seq{\T\x,\x}}^4 &=& \abs{\seq{P\x,\x}+\seq{\x,M^*\x}}^4\leq \bra{\abs{\seq{P\x,\x}}+\abs{\seq{\x,M^*\x}}}^4 \\
   &=&\bra{\abs{\seq{P\x,\x}}^2+\abs{\seq{\x,M^*\x}}^2+2\abs{\seq{P\x,\x}\seq{\x,M^*\x}}}^2\\
   &\leq&\bra{\sqrt{\seq{P\x,P\x}^2+\seq{M^*\x,M^*\x}^2+2\abs{\seq{P\x,M^*\x}}^2}+2\abs{\seq{P\x,\x}\seq{\x,M^*\x}}}^2\\
   &&\bra{\mbox{by Lemma \ref{Drag}}}\\
   &\leq& 2\bra{\seq{P\x,P\x}^2+\seq{M^*\x,M^*\x}^2+2\abs{\seq{P\x,M^*\x}}^2}+8\abs{\seq{P\x,\x}\seq{\x,M^*\x}}^2\\
   &&\bra{\mbox{by the convexity of the function $f(t)=t^2$}}\\
   &\leq& 2\bra{\seq{P\x,P\x}^2+\seq{M^*\x,M^*\x}^2+2\abs{\seq{P\x,M^*\x}}^2}\\
   &+&8\bra{\chi_1\norm{P\x}^2\norm{M^*\x}^2
   +\chi_2\norm{P\x}\norm{M^*\x}\abs{\seq{P\x,M^*\x}}}\\
   &&\bra{\mbox{by Lemma \ref{Mix-al-be}}}\\
   &\leq&2\bra{\seq{\bra{|P|^4+|M^*|^4}\x,\x}+\abs{\seq{MP\x,\x}}^2}\\
   &+& 8\bra{\frac{\chi_1}{2}\seq{\bra{|P|^4+|M^*|^4}\x,\x}+\frac{\chi_2}{2}\seq{\bra{|P|^2+|M^*|^2}\x,\x}\abs{\seq{MP\x,\x}}}\\
   &&\bra{\mbox{by the arithmetic-geometric inequality  mean}}\\
   &=& \bra{2+4\chi_1}\seq{\bra{|P|^4+|M^*|^4}\x,\x}+4\chi_2\seq{\bra{|P|^2+|M^*|^2}\x,\x}\abs{\seq{MP\x,\x}}+2\abs{\seq{MP\x,\x}}^2\\
   &\leq&\bra{2+4\chi_1}\norm{|P|^4+|M^*|^4}+4\chi_2\norm{|P|^2+|M^*|^2}\omega(MP)+2\omega^2(MP).
\end{eqnarray*}
By taking the supremum over all vectors of $\x\in\hh$ and using Lemma \ref{lemma5.1}, we obtain
\begin{eqnarray*}
  \omega^4(\T) &\leq& \bra{2+4\chi_1}\max\set{\norm{|T|^4+|X^*|^4},\norm{|S|^4+|Y^*|^4}}\\
   &+&4\chi_2\max\set{\norm{|T|^2+|X^*|^2},\norm{|S|^2+|Y^*|^2}}\omega(R)+2\omega^2(R).
\end{eqnarray*}
\end{proof}
  Now, we give some special cases of Theorem \ref{KZ}.
  \begin{corollary}\label{cor2.31} Let $T\in\bh$, $\alpha\in\c\setminus\{0\}$ and $\beta\geq 0$. Then
    \begin{equation}\label{KZ1}
      \omega^4(T)\leq \bra{\frac{2+4\chi_1}{16}}\norm{|T|^4+|T^*|^4}+
      \frac{\chi_2}{4}\norm{|T|^2+|T^*|^2}\omega(T^2)+\frac{1}{8}\omega^2(T^2),
    \end{equation}
    where $\chi_1=\frac{\beta+(\beta+1)\max\set{1,\abs{\alpha-1}^2}}{\abs{\alpha}^2(\beta+1)}$ and
$\chi_2=\frac{1+2(\beta+1)\max\set{1,\abs{\alpha-1}}}{\abs{\alpha}^2(\beta+1)}$. In particular,
\begin{equation}\label{KZ2}
      \omega^4(T)\leq \bra{\frac{4\beta+3}{32(\beta+1)}}\norm{|T|^4+|T^*|^4}+\bra{\frac{2\beta+3}{16(\beta+1)}}\norm{|T|^2+|T^*|^2}\omega(T^2)+\frac{1}{8}\omega^2(T^2).
    \end{equation}
  \end{corollary}
  \begin{proof} By Theorem \ref{KZ}, we have
  \begin{eqnarray*}
    \omega^4\bra{\begin{bmatrix} T& T\\ T& T \\\end{bmatrix}} &\leq& \bra{2+4\chi_1}\max\set{\norm{|T|^4+|T^*|^4},\norm{|T|^4+|T^*|^4}}\\
   &+&4\chi_2\max\set{\norm{|T|^2+|T^*|^2},\norm{|T|^2+|T^*|^2}}w\bra{\begin{bmatrix} O& T^2\\ T^2&O\\\end{bmatrix}}+2\omega^2\bra{\begin{bmatrix} O& T^2\\ T^2&O\\\end{bmatrix}}.
  \end{eqnarray*}
  Hence, by Lemma \ref{lemma5.1}, it follows that
\begin{eqnarray*}
  16\omega^4(T) &\leq& \bra{2+4\chi_1}\norm{|T|^4+|T^*|^4}+4\chi_2\norm{|T|^2+|T^*|^2}w\bra{T^2}+2\omega^2\bra{T^2}.
\end{eqnarray*}
which gives the inequality (\ref{KZ1}). The inequality (\ref{KZ2}) also follows from (\ref{KZ1}) by letting $\alpha=2$.
  \end{proof}
\begin{remark} For $T\in\bh$, El-Haddad and Kittaneh in \cite[Theorem 2]{Had-Kit}obtained the following result:
\begin{equation}\label{KZ3}
  \omega^4(T)\leq \frac{1}{2}\norm{|T|^4+|T^*|^4}.
\end{equation}
A refinement of the inequality (\ref{KZ3}) has been established in  \cite[Remark 3.2]{Dom-kit2}. This refinement asserts that
\begin{equation}\label{KZ4}
  \omega^4(T)\leq \frac{3}{8}\norm{|T|^4+|T^*|^4}+\frac{1}{8}\norm{|T|^2+|T^*|^2}\omega(T^2).
\end{equation}
It can be checked that the inequality (\ref{KZ2}) is a refinement of the inequality (\ref{KZ4}) for any $\beta\geq 0$. In fact,
\begin{eqnarray*}
  \omega^4(T) &\leq&\bra{\frac{4\beta+3}{32(\beta+1)}}\norm{|T|^4+|T^*|^4}+\bra{\frac{2\beta+3}{16(\beta+1)}}\norm{|T|^2+|T^*|^2}\omega(T^2)+\frac{1}{8}\omega^2(T^2) \\
  &\leq& \bra{\frac{4\beta+3}{32(\beta+1)}}\norm{|T|^4+|T^*|^4}+\bra{\frac{2\beta+3}{16(\beta+1)}}\norm{|T|^2+|T^*|^2}\omega^2(T)+\frac{1}{8}\omega(T^2)\omega(T^2)\\
  &&\bra{\mbox{by the  power inequality for numerical radius}}\\
  &\leq& \bra{\frac{4\beta+3}{32(\beta+1)}}\norm{|T|^4+|T^*|^4}+\bra{\frac{2\beta+3}{32(\beta+1)}}\norm{|T|^2+|T^*|^2}^2+\frac{1}{16}\norm{|T|^2+|T^*|^2}\omega(T^2)\\
  &&\bra{\mbox{by the  inequality (\ref{Ineq1.2})}}\\
  &\leq&\bra{\frac{4\beta+3}{32(\beta+1)}}\norm{|T|^4+|T^*|^4}+\bra{\frac{2\beta+3}{32(\beta+1)}}\norm{|T|^4+|T^*|^4}+\frac{1}{16}\norm{|T|^2+|T^*|^2}\omega(T^2)\\
  &&\bra{\mbox{by Lemma \ref{lemma2.2} }}\\
  &=&\frac{3}{16}\norm{|T|^4+|T^*|^4}+\frac{1}{16}\norm{|T|^2+|T^*|^2}\omega(T^2)\\
  &<& \frac{3}{8}\norm{|T|^4+|T^*|^4}+\frac{1}{8}\norm{|T|^2+|T^*|^2}\omega(T^2).
\end{eqnarray*}
In addition, it is see that from the above discussion the inequality (\ref{KZ2}) is a refinement of the inequality (\ref{KZ3}). Indeed,
\begin{eqnarray*}
  \omega^4(T) &\leq&\frac{3}{16}\norm{|T|^4+|T^*|^4}+\frac{1}{16}\norm{|T|^2+|T^*|^2}\omega(T^2) \\
  &\leq& \frac{3}{16}\norm{|T|^4+|T^*|^4}+\frac{1}{16}\norm{|T|^2+|T^*|^2}\omega^2(T)\,\bra{\mbox{by the  power inequality for numerical radius}}\\
  &\leq& \frac{3}{16}\norm{|T|^4+|T^*|^4}+\frac{1}{32}\norm{|T|^2+|T^*|^2}^2\bra{\mbox{by the inequality (\ref{Ineq1.2})}}\\
  &\leq&\frac{3}{16}\norm{|T|^4+|T^*|^4}+\frac{1}{32}\norm{|T|^4+|T^*|^4}\\
  &=& \frac{7}{32}\norm{|T|^4+|T^*|^4}\\
  &<& \frac{1}{2}\norm{|T|^4+|T^*|^4}.
\end{eqnarray*}
\end{remark}

\begin{theorem}\label{Modified KZ}  Let $T,S,X,Y\in\bh$, $\mu\in [0,1]$, $\alpha\in\c\setminus\{0\}$ and $\beta\geq 0$. Then
\begin{eqnarray}
 \omega^4\bra{\begin{bmatrix} T& X\\ Y& S \\\end{bmatrix}}
   &\leq& \bra{2+2\chi_1+2\chi_3}\max\set{\norm{|T|^4+|X^*|^4},\norm{|S|^4+|Y^*|^4}}\nonumber\\
     &+& (2\chi_2+2\mu\chi_4)\max\set{\norm{|T|^2+|X^*|^2},\norm{|S|^2+|Y^*|^2}}\omega\bra{\begin{bmatrix} O& XS\\ YT& O \\\end{bmatrix}}\nonumber\\
     &+& \bra{4+4(1-\mu)\chi_4}\omega^2\bra{\begin{bmatrix} O& XS\\ YT& O \\\end{bmatrix}},
\end{eqnarray}
 where $\chi_1=\frac{\beta+(\beta+1)\max\set{1,\abs{\alpha-1}^2}}{\abs{\alpha}^2(\beta+1)}$,
$\chi_2=\frac{1+2(\beta+1)\max\set{1,\abs{\alpha-1}}}{\abs{\alpha}^2(\beta+1)}$,
 $\chi_3=\frac{2\beta+2(\beta+1)\max\set{1,\abs{\alpha-1}^2}}{\abs{\alpha}^2(\beta+1)}$
 and $\chi_4=\frac{2}{\abs{\alpha}^2(\beta+1)}$.
\end{theorem}
\begin{proof} Let \(\mathbf{x} = \begin{bmatrix} x_1 \\ x_2 \end{bmatrix} \in \mathcal{H} \oplus \mathcal{H}\) be a unit vector. Define:
\[
\mathbb{T} = \begin{bmatrix} T & X \\ Y & S \end{bmatrix}, \quad \mathbb{M} = \begin{bmatrix} O & X \\ Y & O \end{bmatrix}, \quad \mathbb{P} = \begin{bmatrix} T & O \\ O & S \end{bmatrix}, \quad \mathbb{E} = \begin{bmatrix} O & XS \\ YT & O \end{bmatrix}.
\]
Note that \(\mathbb{M}\mathbb{P} = \mathbb{E}\), and
\[
|\mathbb{P}|^{4} + |\mathbb{M}^{*}|^{4} = \begin{bmatrix} |T|^{4} + |X^{*}|^{4} & O \\ O & |S|^{4} + |Y^{*}|^{4} \end{bmatrix}, \quad
|\mathbb{P}|^{2} + |\mathbb{M}^{*}|^{2} = \begin{bmatrix} |T|^{2} + |X^{*}|^{2} & O \\ O & |S|^{2} + |Y^{*}|^{2} \end{bmatrix}.
\]

From the proof of Theorem \ref{KZ}, we have
\[
|\langle \mathbb{T} \mathbf{x}, \mathbf{x} \rangle|^{4} \leq 2 \left( \langle \mathbb{P} \mathbf{x}, \mathbb{P} \mathbf{x} \rangle^{2} + \langle \mathbb{M}^{*} \mathbf{x}, \mathbb{M}^{*} \mathbf{x} \rangle^{2} + 2 |\langle \mathbb{P} \mathbf{x}, \mathbb{M}^{*} \mathbf{x} \rangle|^{2} \right) + 8 |\langle \mathbb{P} \mathbf{x}, \mathbf{x} \rangle \langle \mathbf{x}, \mathbb{M}^{*} \mathbf{x} \rangle|^{2}.
\]

Split the last term as
\[
8 |\langle \mathbb{P} \mathbf{x}, \mathbf{x} \rangle \langle \mathbf{x}, \mathbb{M}^{*} \mathbf{x} \rangle|^{2} = 4 |\langle \mathbb{P} \mathbf{x}, \mathbf{x} \rangle \langle \mathbf{x}, \mathbb{M}^{*} \mathbf{x} \rangle|^{2} + 4 |\langle \mathbb{P} \mathbf{x}, \mathbf{x} \rangle \langle \mathbf{x}, \mathbb{M}^{*} \mathbf{x} \rangle|^{2}.
\]

Apply Lemma \ref{Mix-al-be}  to the first part, we get
\[
4 |\langle \mathbb{P} \mathbf{x}, \mathbf{x} \rangle \langle \mathbf{x}, \mathbb{M}^{*} \mathbf{x} \rangle|^{2} \leq 4 \left( \chi_{1} \| \mathbb{P} \mathbf{x} \|^{2} \| \mathbb{M}^{*} \mathbf{x} \|^{2} + \chi_{2} \| \mathbb{P} \mathbf{x} \| \| \mathbb{M}^{*} \mathbf{x} \| |\langle \mathbb{P} \mathbf{x}, \mathbb{M}^{*} \mathbf{x} \rangle| \right).
\]

Apply Lemma \ref{Ramadan-Kareem1} to the second part, we obtain
\[
4 |\langle \mathbb{P} \mathbf{x}, \mathbf{x} \rangle \langle \mathbf{x}, \mathbb{M}^{*} \mathbf{x} \rangle|^{2} \leq 4 \left( \chi_{3} \| \mathbb{P} \mathbf{x} \|^{2} \| \mathbb{M}^{*} \mathbf{x} \|^{2} + \chi_{4} |\langle \mathbb{P} \mathbf{x}, \mathbb{M}^{*} \mathbf{x} \rangle|^{2} \right).
\]
Since \(\chi_3 = 2\chi_1\), this becomes
\[
4 \left( 2\chi_{1} \| \mathbb{P} \mathbf{x} \|^{2} \| \mathbb{M}^{*} \mathbf{x} \|^{2} + \chi_{4} |\langle \mathbb{P} \mathbf{x}, \mathbb{M}^{*} \mathbf{x} \rangle|^{2} \right).
\]

Combine both parts implies
\begin{eqnarray*}
  8 |\langle \mathbb{P} \mathbf{x}, \mathbf{x} \rangle \langle \mathbf{x}, \mathbb{M}^{*} \mathbf{x} \rangle|^{2}  &\leq&
  4\chi_{1} \| \mathbb{P} \mathbf{x} \|^{2} \| \mathbb{M}^{*} \mathbf{x} \|^{2} + 4\chi_{2} \| \mathbb{P} \mathbf{x} \| \| \mathbb{M}^{*} \mathbf{x} \| |\langle \mathbb{E} \mathbf{x}, \mathbf{x} \rangle| \\
   &+& 8\chi_{1} \| \mathbb{P} \mathbf{x} \|^{2} \| \mathbb{M}^{*} \mathbf{x} \|^{2} + 4\chi_{4} |\langle \mathbb{E} \mathbf{x}, \mathbf{x} \rangle|^{2}.
\end{eqnarray*}
Simplify, we obtain
\[
= (12\chi_{1}) \| \mathbb{P} \mathbf{x} \|^{2} \| \mathbb{M}^{*} \mathbf{x} \|^{2} + 4\chi_{2} \| \mathbb{P} \mathbf{x} \| \| \mathbb{M}^{*} \mathbf{x} \| |\langle \mathbb{E} \mathbf{x}, \mathbf{x} \rangle| + 4\chi_{4} |\langle \mathbb{E} \mathbf{x}, \mathbf{x} \rangle|^{2}.
\]

Decompose the last term using \(\mu\) yields
\[
4\chi_{4} |\langle \mathbb{E} \mathbf{x}, \mathbf{x} \rangle|^{2} = 4\mu\chi_{4} |\langle \mathbb{E} \mathbf{x}, \mathbf{x} \rangle|^{2} + 4(1-\mu)\chi_{4} |\langle \mathbb{E} \mathbf{x}, \mathbf{x} \rangle|^{2}.
\]
By the Cauchy-Schwarz inequality, we have
\[
|\langle \mathbb{P} \mathbf{x}, \mathbb{M}^{*} \mathbf{x} \rangle| \leq \| \mathbb{P} \mathbf{x} \| \| \mathbb{M}^{*} \mathbf{x} \|.
\]
So,
\[
4\mu\chi_{4} |\langle \mathbb{E} \mathbf{x}, \mathbf{x} \rangle|^{2} \leq 4\mu\chi_{4} \| \mathbb{P} \mathbf{x} \| \| \mathbb{M}^{*} \mathbf{x} \| \cdot |\langle \mathbb{E} \mathbf{x}, \mathbf{x} \rangle|.
\]
Substitute, we obtain
\begin{eqnarray*}
  8 |\langle \mathbb{P} \mathbf{x}, \mathbf{x} \rangle \langle \mathbf{x}, \mathbb{M}^{*} \mathbf{x} \rangle|^{2} &\leq&
  12\chi_{1} \| \mathbb{P} \mathbf{x} \|^{2} \| \mathbb{M}^{*} \mathbf{x} \|^{2} + 4\chi_{2} \| \mathbb{P} \mathbf{x} \| \| \mathbb{M}^{*} \mathbf{x} \| |\langle \mathbb{E} \mathbf{x}, \mathbf{x} \rangle| \\
   &+& 4\mu\chi_{4} \| \mathbb{P} \mathbf{x} \| \| \mathbb{M}^{*} \mathbf{x} \| |\langle \mathbb{E} \mathbf{x}, \mathbf{x} \rangle| + 4(1-\mu)\chi_{4} |\langle \mathbb{E} \mathbf{x}, \mathbf{x} \rangle|^{2}.
\end{eqnarray*}
Simplify, we get
\[
= 12\chi_{1} \| \mathbb{P} \mathbf{x} \|^{2} \| \mathbb{M}^{*} \mathbf{x} \|^{2} + 4(\chi_{2} + \mu\chi_{4}) \| \mathbb{P} \mathbf{x} \| \| \mathbb{M}^{*} \mathbf{x} \| |\langle \mathbb{E} \mathbf{x}, \mathbf{x} \rangle| + 4(1-\mu)\chi_{4} |\langle \mathbb{E} \mathbf{x}, \mathbf{x} \rangle|^{2}.
\]

Now substitute into the main inequality, we have
\begin{eqnarray*}
  |\langle \mathbb{T} \mathbf{x}, \mathbf{x} \rangle|^{4}  &\leq&2 \left( \langle \mathbb{P} \mathbf{x}, \mathbb{P} \mathbf{x} \rangle^{2} + \langle \mathbb{M}^{*} \mathbf{x}, \mathbb{M}^{*} \mathbf{x} \rangle^{2}+ 2 |\langle \mathbb{E} \mathbf{x}, \mathbf{x} \rangle|^{2} \right)\\
   &+& 12\chi_{1} \| \mathbb{P} \mathbf{x} \|^{2} \| \mathbb{M}^{*} \mathbf{x} \|^{2} +4(\chi_{2}+ \mu\chi_{4}) \| \mathbb{P} \mathbf{x} \| \| \mathbb{M}^{*} \mathbf{x} \| |\langle \mathbb{E} \mathbf{x}, \mathbf{x} \rangle|\\
   &+&4(1-\mu)\chi_{4} |\langle \mathbb{E} \mathbf{x}, \mathbf{x} \rangle|^{2}.
\end{eqnarray*}
Use Lemma \ref{Holder}(i) for \(r=2\), we obtain
\[
\langle \mathbb{P} \mathbf{x}, \mathbb{P} \mathbf{x} \rangle^{2} = \langle |\mathbb{P}|^{2} \mathbf{x}, \mathbf{x} \rangle^{2} \leq \langle |\mathbb{P}|^{4} \mathbf{x}, \mathbf{x} \rangle, \quad \langle \mathbb{M}^{*} \mathbf{x}, \mathbb{M}^{*} \mathbf{x} \rangle^{2} = \langle |\mathbb{M}^{*}|^{2} \mathbf{x}, \mathbf{x} \rangle^{2} \leq \langle |\mathbb{M}^{*}|^{4} \mathbf{x}, \mathbf{x} \rangle.
\]
Apply the arithmetic-geometric mean inequality, we get
\[
\| \mathbb{P} \mathbf{x} \|^{2} \| \mathbb{M}^{*} \mathbf{x} \|^{2} \leq \frac{1}{2} \left( \langle |\mathbb{P}|^{4} \mathbf{x}, \mathbf{x} \rangle + \langle |\mathbb{M}^{*}|^{4} \mathbf{x}, \mathbf{x} \rangle \right).
\]
Thus,
\[
12\chi_{1} \| \mathbb{P} \mathbf{x} \|^{2} \| \mathbb{M}^{*} \mathbf{x} \|^{2} \leq 6\chi_{1} \left( \langle |\mathbb{P}|^{4} \mathbf{x}, \mathbf{x} \rangle + \langle |\mathbb{M}^{*}|^{4} \mathbf{x}, \mathbf{x} \rangle \right).
\]
Also,
\[
4(\chi_{2} + \mu\chi_{4}) \| \mathbb{P} \mathbf{x} \| \| \mathbb{M}^{*} \mathbf{x} \| |\langle \mathbb{E} \mathbf{x}, \mathbf{x} \rangle| \leq 2(\chi_{2} + \mu\chi_{4}) \langle (|\mathbb{P}|^{2} + |\mathbb{M}^{*}|^{2}) \mathbf{x}, \mathbf{x} \rangle |\langle \mathbb{E} \mathbf{x}, \mathbf{x} \rangle|,
\]
since \(\| \mathbb{P} \mathbf{x} \| \| \mathbb{M}^{*} \mathbf{x} \| \leq \frac{1}{2} \langle (|\mathbb{P}|^{2} + |\mathbb{M}^{*}|^{2}) \mathbf{x}, \mathbf{x} \rangle\). Combine all terms, we have
\begin{eqnarray*}
  |\langle \mathbb{T} \mathbf{x}, \mathbf{x} \rangle|^{4} &\leq&(2 + 2\chi_{1} + 2\chi_{3}) \langle (|\mathbb{P}|^{4} + |\mathbb{M}^{*}|^{4}) \mathbf{x}, \mathbf{x} \rangle \\
   &+& (2\chi_{2} + 2\mu\chi_{4}) \langle (|\mathbb{P}|^{2} + |\mathbb{M}^{*}|^{2}) \mathbf{x}, \mathbf{x} \rangle |\langle \mathbb{E} \mathbf{x}, \mathbf{x} \rangle|
+ (4 + 4(1-\mu)\chi_{4}) |\langle \mathbb{E} \mathbf{x}, \mathbf{x} \rangle|^{2},
\end{eqnarray*}
where \(\chi_3 = 2\chi_1\) is used.

By taking the supremum over all vectors of $\x\in\hh$ and using Lemma \ref{lemma5.1}, we obtain
\begin{eqnarray*}
 \langle (|\mathbb{P}|^{4} + |\mathbb{M}^{*}|^{4}) \mathbf{x}, \mathbf{x} \rangle&\leq&\max\left\{ \||T|^{4} + |X^{*}|^{4}\|, \||S|^{4} + |Y^{*}|^{4}\| \right\}, \\
\langle (|\mathbb{P}|^{2} + |\mathbb{M}^{*}|^{2}) \mathbf{x}, \mathbf{x} \rangle &\leq& \max\left\{ \||T|^{2} + |X^{*}|^{2}\|, \||S|^{2} + |Y^{*}|^{2}\| \right\},\\
|\langle \mathbb{E} \mathbf{x}, \mathbf{x} \rangle| &\leq& \omega(\mathbb{E}),\,\,\mbox{and}\\
 |\langle \mathbb{E} \mathbf{x}, \mathbf{x} \rangle|^{2} &\leq& \omega^{2}(\mathbb{E}).
\end{eqnarray*}
This yields the desired inequality.
\end{proof}
The following Example illustrating  the validity of Theorem \ref{Modified KZ}.
\begin{example} To verify Theorem 2.33 for the operator matrix
\[
A = \begin{bmatrix} T & X \\ Y & S \end{bmatrix} = \begin{bmatrix} \begin{pmatrix} 1 & -1 \\ 0 & 1 \end{pmatrix} & \begin{pmatrix} 2 & 1 \\ -1 & 1 \end{pmatrix} \\ \begin{pmatrix} 1 & -1 \\ 1 & 1 \end{pmatrix} & \begin{pmatrix} 1 & 1 \\ 2 & 1 \end{pmatrix} \end{bmatrix},
\]
with parameters \(\alpha = 2\), \(\beta = 0\), and \(\mu = 0.5\), the coefficients are calculated as:
\[
\chi_1 = \frac{1}{4}, \quad \chi_2 = \frac{3}{4}, \quad \chi_3 = \frac{1}{2}, \quad \chi_4 = \frac{1}{2}.
\]

\begin{enumerate}
  \item [Step 1:]Compute Required Norms: For \(T\) and \(X^*\):
   \[
  \| |T|^4 + |X^*|^4 \| = \left\| \begin{pmatrix} 28 & 4 \\ 4 & 10 \end{pmatrix} \right\| \approx 28.8489.
  \]
  For \(S\) and \(Y^*\):
  \[
  \| |S|^4 + |Y^*|^4 \| = \left\| \begin{pmatrix} 38 & 21 \\ 21 & 17 \end{pmatrix} \right\| \approx 50.978.
  \]
  \[
  \max\left\{ \| |T|^4 + |X^*|^4 \|, \| |S|^4 + |Y^*|^4 \| \right\} = 50.978.
  \]
  For \(\| |T|^2 + |X^*|^2 \|\) and \(\| |S|^2 + |Y^*|^2 \|\):
  \[
  \| |T|^2 + |X^*|^2 \| = \left\| \begin{pmatrix} 6 & 0 \\ 0 & 4 \end{pmatrix} \right\| = 6, \quad \| |S|^2 + |Y^*|^2 \| = \left\| \begin{pmatrix} 7 & 3 \\ 3 & 4 \end{pmatrix} \right\| \approx 8.854.
  \]
  \[
  \max\left\{ \| |T|^2 + |X^*|^2 \|, \| |S|^2 + |Y^*|^2 \| \right\} = 8.854.
  \]
  \item [Step 2:] Compute Numerical Radius of Off-Diagonal Operator:
The off-diagonal operator is
\[
E = \begin{bmatrix} O & XS \\ YT & O \end{bmatrix}, \quad \text{where} \quad XS = \begin{pmatrix} 4 & 3 \\ 1 & 0 \end{pmatrix}, \quad YT = \begin{pmatrix} 1 & -2 \\ 1 & 0 \end{pmatrix}.
\]
The matrix form is
\[
E = \begin{pmatrix} 0 & 0 & 4 & 3 \\ 0 & 0 & 1 & 0 \\ 1 & -2 & 0 & 0 \\ 1 & 0 & 0 & 0 \end{pmatrix}.
\]
The numerical radius is \(\omega(E) \approx 1.6884\), so \(\omega^2(E) \approx 2.8508\).
  \item [Step 3:] Compute the Bound:\\
Using the coefficients:
\[
2 + 2\chi_1 + 2\chi_3 = 3.5, \quad 2\chi_2 + 2\mu\chi_4 = 2, \quad 4 + 4(1-\mu)\chi_4 = 5.
\]
The bound is:
\[
(3.5 \times 50.978) + (2 \times 8.854 \times 1.6884) + (5 \times 2.8508) \approx 178.423 + 29.9 + 14.254 = 222.577.
\]
\item [ Step 4:]  Compute \(\omega^4(A)\)
The numerical radius of \(A\) is \(\omega(A) \approx 3.732\), so
\[
\omega^4(A) \approx (3.732)^4 \approx 193.8.
\]
Since \(193.8 \leq 222.577\), the inequality holds:
\[
\omega^4(A) \leq 222.577.
\]
\end{enumerate}
\end{example}
We applying Theorem \ref{Modified KZ} to some special cases.
\begin{corollary}\label{Cor3.27} Let $T\in\bh$,$\mu\in [0,1]$, $\alpha\in\c\setminus\{0\}$ and $\beta\geq 0$. Then
  \begin{eqnarray}\label{PVS1}
  \omega^4\bra{T}
   &\leq& \bra{\frac{1+\chi_1+\chi_3}{8}}\norm{|T|^4+|T^*|^4}+ \bra{\frac{(\chi_2+\mu\chi_4)}{8}}\norm{|T|^2+|T^*|^2}\omega\bra{T^2}\nonumber\\
     &+& \bra{\frac{1+(1-\mu)\chi_4}{4}}\omega^2\bra{T^2},
\end{eqnarray}
 where $\chi_1=\frac{\beta+(\beta+1)\max\set{1,\abs{\alpha-1}^2}}{\abs{\alpha}^2(\beta+1)}$,
$\chi_2=\frac{1+2(\beta+1)\max\set{1,\abs{\alpha-1}}}{\abs{\alpha}^2(\beta+1)}$,
 $\chi_3=\frac{2\beta+2(\beta+1)\max\set{1,\abs{\alpha-1}^2}}{\abs{\alpha}^2(\beta+1)}$
 and $\chi_4=\frac{2}{\abs{\alpha}^2(\beta+1)}$. In particular,
 \begin{eqnarray}\label{PVS2}
  \omega^4\bra{T} &\leq&\bra{\frac{3\beta+2}{16(\beta+1)}}\norm{|T|^4+|T^*|^4}+ \bra{\frac{2\beta+2\mu+3}{32(\beta+1)}} \norm{|T|^2+|T^*|^2}\omega\bra{T^2}\nonumber\\
     &+& \bra{\frac{3-2\mu}{16(\beta+1)}}\omega^2\bra{T^2}.
\end{eqnarray}
\end{corollary}
\begin{proof} With the exception of using Theorem \ref{Modified KZ} in place of Theorem \ref{KZ}, the proof is comparable to the proof of Corollary \ref{KZ1}.
\end{proof}
\begin{remark} Let \( T \in \mathcal{L}(\mathcal{H}) \), \(\mu \in [0,1]\), and \(\beta \geq 0\). The inequality (2.24) holds:
\begin{eqnarray}\label{ineq2.24}
 \omega^{4}(T)&\leq&\frac{3\beta+2}{16(\beta+1)} \left\||T|^{4}+|T^{*}|^{4}\right\| + \frac{2\beta+2\mu+3}{32(\beta+1)} \left\||T|^{2}+|T^{*}|^{2}\right\| \omega(T^{2})\nonumber \\
   &+& \frac{3-2\mu}{16(\beta+1)} \omega^{2}(T^{2}).
\end{eqnarray}
When \(\alpha = 2\), (\ref{ineq2.24}) reduces to (\ref{KZ1}):
\[
\omega^{4}(T) \leq \frac{4\beta+3}{32(\beta+1)} \left\||T|^{4}+|T^{*}|^{4}\right\| + \frac{2\beta+3}{16(\beta+1)} \left\||T|^{2}+|T^{*}|^{2}\right\| \omega(T^{2}) + \frac{1}{8} \omega^{2}(T^{2}).
\]
As proven in Corollary \ref{cor2.31}, (\ref{KZ1}) refines both (\ref{KZ2}) and (\ref{KZ3}). Specifically, the proof shows
\begin{enumerate}
  \item [(i)] (\ref{KZ1}) refines (\ref{KZ3}) for all \(\beta \geq 0\).
  \item [(ii)] (\ref{KZ1}) refines (\ref{KZ2}) for all \(\beta \geq 0\).
\end{enumerate}
Thus, (\ref{ineq2.24}) extends these refinements to general \(\alpha \in \mathbb{C} \setminus \{0\}\), with the case \(\alpha = 2\) explicitly verified.
\end{remark}
In order to prove Theorem \ref{SEEMA-9}, we need the following lemma ( A more generalized version of Lemma \ref{Ramadan-Kareem1} ).
\begin{lemma}\label{GREAT-A1} Let $a,b,e\in\h$ with $\norm{e}=1$. Then  for all $\beta\geq 0$, $\alpha\in \c\setminus\{0\}$ and \(n \in \mathbb{N}\)
\begin{eqnarray*}
 \abs{\seq{a,e}\seq{e,b}}^{2n}&\leq&\sum_{k=0}^{n}\binom{n}{k}\lambda_1^{k}\lambda_2^{n-k}\norm{a}^{2k}
 \norm{b}^{2k}\norm{a}^{n-k}\norm{b}^{n-k}\abs{\seq{a,b}}^{n-k}\\
 &=& \lambda_1^n\norm{a}^{2n}\norm{b}^{2n}+\sum_{k=1}^{n-1}\binom{n}{k}\lambda_1^{k}\lambda_2^{n-k}\norm{a}^{2k}
 \norm{b}^{2k}\norm{a}^{n-k}\norm{b}^{n-k}\abs{\seq{a,b}}^{n-k}\\
 &+&\lambda_2^n\norm{a}^n\norm{b}^n\abs{\seq{a,b}}^{n},
\end{eqnarray*}
 where $\lambda_1=\frac{\beta+(\beta+1)\max\set{1,\abs{\alpha-1}^2}}{\abs{\alpha}^2(\beta+1)}$ and
$\lambda_2=\frac{1+2(\beta+1)\max\set{1,\abs{\alpha-1}}}{\abs{\alpha}^2(\beta+1)}$.
\end{lemma}
\begin{proof} Let \(a, b, e \in \mathcal{H}\) with \(\|e\| = 1\). By Lemma \ref{Mix-al-be} (which holds for \(\alpha \in \mathbb{C} \setminus \{0\}\) and \(\beta \geq 0\)), we have
\[
|\langle a, e \rangle \langle e, b \rangle|^2 \leq \lambda_1 \|a\|^2 \|b\|^2 + \lambda_2 \|a\| \|b\| |\langle a, b \rangle|,
\]
where \(\lambda_1 = \frac{\beta + (\beta+1) \max\{1, |\alpha-1|^2\}}{|\alpha|^2 (\beta+1)}\) and \(\lambda_2 = \frac{1 + 2(\beta+1) \max\{1, |\alpha-1|\}}{|\alpha|^2 (\beta+1)}\). Since both sides are nonnegative, raising to the \(n\)-th power preserves the inequality:
\[
|\langle a, e \rangle \langle e, b \rangle|^{2n} \leq \left( \lambda_1 \|a\|^2 \|b\|^2 + \lambda_2 \|a\| \|b\| |\langle a, b \rangle| \right)^n.
\]
Expanding the right-hand side using the binomial theorem, we get
\[
\left( \lambda_1 \|a\|^2 \|b\|^2 + \lambda_2 \|a\| \|b\| |\langle a, b \rangle| \right)^n = \sum_{k=0}^{n} \binom{n}{k} \left( \lambda_1 \|a\|^2 \|b\|^2 \right)^k \left( \lambda_2 \|a\| \|b\| |\langle a, b \rangle| \right)^{n-k}.
\]
Simplifying each term, we have
\[
\left( \lambda_1 \|a\|^2 \|b\|^2 \right)^k = \lambda_1^k \|a\|^{2k} \|b\|^{2k}, \quad \left( \lambda_2 \|a\| \|b\| |\langle a, b \rangle| \right)^{n-k} = \lambda_2^{n-k} \|a\|^{n-k} \|b\|^{n-k} |\langle a, b \rangle|^{n-k}.
\]
Combining these, the expression becomes
\[
\abs{\seq{a,e}\seq{e,b}}^{2n}\leq \sum_{k=0}^{n} \binom{n}{k} \lambda_1^k \lambda_2^{n-k} \|a\|^{2k} \|b\|^{2k} \|a\|^{n-k} \|b\|^{n-k} |\langle a, b \rangle|^{n-k}.
\]
Thus, the inequality holds as stated.
\end{proof}
We are now ready to demonstrate our subsequent finding.
\begin{theorem}\label{SEEMA-9}  Let \(B, C \in \mathcal{L}(\mathcal{H})\). Then for every \(\mu \in [0, 1]\), $\beta\geq 0$, $\alpha\in \c\setminus\{0\}$ and \(n \in \mathbb{N}\) with \(n \geq 2\),
\begin{eqnarray}
  \omega^{2n}\left(\begin{bmatrix}O & B \\ C & O\end{bmatrix}\right) &\leq&\frac{\lambda_1^n}{2} \max\left\{ \left\| |C|^{4n\mu} + |B^*|^{4n(1-\mu)} \right\|, \left\| |B|^{4n\mu} + |C^*|^{4n(1-\mu)} \right\| \right\}\nonumber \\
   &+&\lambda_2^n \max\left\{ \omega^n \left( |B^*|^{2(1-\mu)} |C|^{2\mu} \right), \omega^n \left( |C^*|^{2(1-\mu)} |B|^{2\mu} \right) \right\}\nonumber\\
&+& \frac{1}{2} \sum_{k=1}^{n-1} \binom{n}{k} \lambda_1^{n-k}\lambda_2^k \max\left\{ \left\| |C|^{4(n-k)\mu} + |B^*|^{4(n-k)(1-\mu)} \right\|,\right.
\nonumber\\ &&\left.\left\| |B|^{4(n-k)\mu} + |C^*|^{4(n-k)(1-\mu)} \right\| \right\}\\
&\times&\max\left\{ \omega^k \left( |B^*|^{2(1-\mu)} |C|^{2\mu} \right), \omega^k \left( |C^*|^{2(1-\mu)} |B|^{2\mu} \right) \right\},\nonumber
\end{eqnarray}
where \(\lambda_1 = \frac{\beta + (\beta + 1) \max\{1, |\alpha - 1|^2\}}{|\alpha|^2 (\beta + 1)}\) and \(\lambda_2 = \frac{1 + 2(\beta + 1) \max\{1, |\alpha - 1|\}}{|\alpha|^2 (\beta + 1)}\) for \(\alpha \in \mathbb{C} \setminus \{0\}\).
\end{theorem}
\begin{proof} Let \(\mathbb{T} = \begin{bmatrix} O & B \\ C & O \end{bmatrix}\) and \(\mathbf{x}\) be a unit vector in \(\mathcal{H} \oplus \mathcal{H}\). By Lemma \ref{Lem2.3} and the Cauchy-Schwarz inequality,
\[
|\langle \mathbb{T} \mathbf{x}, \mathbf{x} \rangle| = |\langle |\mathbb{T}| \mathbf{x}, |\mathbb{T}^*| \mathbf{x} \rangle| \leq \| |\mathbb{T}|^\mu \mathbf{x} \| \| |\mathbb{T}^*|^{1-\mu} \mathbf{x} \| = \sqrt{\langle |\mathbb{T}|^{2\mu} \mathbf{x}, \mathbf{x} \rangle} \sqrt{\langle |\mathbb{T}^*|^{2(1-\mu)} \mathbf{x}, \mathbf{x} \rangle}.
\]
Raising both sides to the power \(2n\), we obtain
\[
|\langle \mathbb{T} \mathbf{x}, \mathbf{x} \rangle|^{2n} \leq \left( \langle |\mathbb{T}|^{2\mu} \mathbf{x}, \mathbf{x} \rangle \langle |\mathbb{T}^*|^{2(1-\mu)} \mathbf{x}, \mathbf{x} \rangle \right)^n = \langle |\mathbb{T}|^{2\mu} \mathbf{x}, \mathbf{x} \rangle^n \langle |\mathbb{T}^*|^{2(1-\mu)} \mathbf{x}, \mathbf{x} \rangle^n.
\]
Set \(a = |\mathbb{T}|^\mu \mathbf{x}\), \(b = |\mathbb{T}^*|^{1-\mu} \mathbf{x}\), and \(e = \mathbf{x}\) (unit vector). By Lemma \ref{GREAT-A1}, we have
\[
|\langle a, e \rangle \langle e, b \rangle|^{2n} = |\langle |\mathbb{T}|^\mu \mathbf{x}, \mathbf{x} \rangle \langle \mathbf{x}, |\mathbb{T}^*|^{1-\mu} \mathbf{x} \rangle|^{2n} \leq \sum_{k=0}^{n} \binom{n}{k} \lambda_1^k \lambda_2^{n-k} \|a\|^{n+k} \|b\|^{n+k} |\langle a, b \rangle|^{n-k}.
\]
Recognize that
\[
|\langle \mathbb{T} \mathbf{x}, \mathbf{x} \rangle|^{2n} \leq \sum_{k=0}^{n} \binom{n}{k} \lambda_1^k \lambda_2^{n-k} \|a\|^{n+k} \|b\|^{n+k} |\langle a, b \rangle|^{n-k}.
\]
Now bound each term, For the term $k=n$, we have
\[
\lambda_1^n \|a\|^{2n} \|b\|^{2n} |\langle a, b \rangle|^0 = \lambda_1^n \|a\|^{2n} \|b\|^{2n}.
\]
By Lemma \ref{Holder} (i) with $r=2$, we have
\[
\|a\|^{2n} = \langle |\mathbb{T}|^{2\mu} \mathbf{x}, \mathbf{x} \rangle^n \leq \langle |\mathbb{T}|^{2n \cdot 2\mu} \mathbf{x}, \mathbf{x} \rangle = \langle |\mathbb{T}|^{4n\mu} \mathbf{x}, \mathbf{x} \rangle,
\]
\[
\|b\|^{2n} \leq \langle |\mathbb{T}^*|^{4n(1-\mu)} \mathbf{x}, \mathbf{x} \rangle.
\]
Using the inequality \(uv \leq \frac{1}{2}(u^2 + v^2)\), we get
\[
\|a\|^{2n} \|b\|^{2n} \leq \frac{1}{2} \left( \langle |\mathbb{T}|^{4n\mu} \mathbf{x}, \mathbf{x} \rangle + \langle |\mathbb{T}^*|^{4n(1-\mu)} \mathbf{x}, \mathbf{x} \rangle \right).
\]
Thus,
\[
\lambda_1^n \|a\|^{2n} \|b\|^{2n} \leq \frac{\lambda_1^n}{2} \left( \langle |\mathbb{T}|^{4n\mu} \mathbf{x}, \mathbf{x} \rangle + \langle |\mathbb{T}^*|^{4n(1-\mu)} \mathbf{x}, \mathbf{x} \rangle \right).
\]
For the term $k=0$, we have
\[
\lambda_2^n \|a\|^{n} \|b\|^{n} |\langle a, b \rangle|^{n} = \lambda_2^n |\langle |\mathbb{T}|^\mu \mathbf{x}, |\mathbb{T}^*|^{1-\mu} \mathbf{x} \rangle|^n.
\]
By Lemma \ref{lemma5.1} and the definition of numerical radius, we obtain
\[
|\langle a, b \rangle| = |\langle |\mathbb{T}|^\mu \mathbf{x}, |\mathbb{T}^*|^{1-\mu} \mathbf{x} \rangle| \leq \omega \left( |\mathbb{T}|^\mu |\mathbb{T}^*|^{1-\mu} \right).
\]
Since \(|\mathbb{T}|^\mu |\mathbb{T}^*|^{1-\mu} = \begin{bmatrix} |C|^\mu |B^*|^{1-\mu} & 0 \\ 0 & |B|^\mu |C^*|^{1-\mu} \end{bmatrix}\),
\[
\omega \left( |\mathbb{T}|^\mu |\mathbb{T}^*|^{1-\mu} \right) = \max\left\{ \omega \left( |B^*|^{1-\mu} |C|^\mu \right), \omega \left( |C^*|^{1-\mu} |B|^\mu \right) \right\}.
\]
Thus,
\[
\lambda_2^n |\langle a, b \rangle|^n \leq \lambda_2^n \left[ \max\left\{ \omega \left( |B^*|^{1-\mu} |C|^\mu \right), \omega \left( |C^*|^{1-\mu} |B|^\mu \right) \right\} \right]^n.
\]
For the terms $1\leq k\leq n-1$, we have
\[
\sum_{k=1}^{n-1}\binom{n}{k}\lambda_1^{k}\lambda_2^{n-k}\norm{a}^{2k}
 \norm{b}^{2k}\norm{a}^{n-k}\norm{b}^{n-k}\abs{\seq{a,b}}^{n-k}=\sum_{k=1}^{n-1}\binom{n}{k} \lambda_1^k \lambda_2^{n-k} \|a\|^{n+k} \|b\|^{n+k} |\langle a, b \rangle|^{n-k}.
\]
By similar reasoning,
\[
\|a\|^{n+k} \|b\|^{n+k} \leq \frac{1}{2} \left( \langle |\mathbb{T}|^{2(n+k)\mu} \mathbf{x}, \mathbf{x} \rangle + \langle |\mathbb{T}^*|^{2(n+k)(1-\mu)} \mathbf{x}, \mathbf{x} \rangle \right),
\]
\[
|\langle a, b \rangle|^{n-k} \leq \left[ \max\left\{ \omega \left( |B^*|^{1-\mu} |C|^\mu \right), \omega \left( |C^*|^{1-\mu} |B|^\mu \right) \right\} \right]^{n-k}.
\]
Thus,
\begin{eqnarray*}
  \binom{n}{k} \lambda_1^k \lambda_2^{n-k} \|a\|^{n+k} \|b\|^{n+k} |\langle a, b \rangle|^{n-k} &\leq&\frac{\binom{n}{k} \lambda_1^k \lambda_2^{n-k}}{2} \left( \langle |\mathbb{T}|^{2(n+k)\mu} \mathbf{x}, \mathbf{x} \rangle+\langle |\mathbb{T}^*|^{2(n+k)(1-\mu)} \mathbf{x}, \mathbf{x} \rangle \right)  \\
   && \times \left[ \max\left\{ \omega \left( |B^*|^{1-\mu} |C|^\mu \right), \omega \left( |C^*|^{1-\mu} |B|^\mu \right) \right\} \right]^{n-k}.
\end{eqnarray*}

Combining all terms, we have
\begin{eqnarray*}
  |\langle \mathbb{T} \mathbf{x}, \mathbf{x} \rangle|^{2n} &\leq&\frac{\lambda_1^n}{2} \left( \langle |\mathbb{T}|^{4n\mu} \mathbf{x}, \mathbf{x} \rangle + \langle |\mathbb{T}^*|^{4n(1-\mu)} \mathbf{x}, \mathbf{x} \rangle \right) \\
   &+& \lambda_2^n \left[ \max\left\{ \omega \left( |B^*|^{1-\mu} |C|^\mu \right), \omega \left( |C^*|^{1-\mu} |B|^\mu \right) \right\} \right]^n\\
&+& \frac{1}{2} \sum_{k=1}^{n-1} \binom{n}{k} \lambda_1^{n-k} \lambda_2^k \left( \langle |\mathbb{T}|^{2(n+k)\mu} \mathbf{x}, \mathbf{x} \rangle + \langle |\mathbb{T}^*|^{2(n+k)(1-\mu)} \mathbf{x}, \mathbf{x} \rangle \right)\\
&\times&\left[ \max\left\{ \omega \left( |B^*|^{1-\mu} |C|^\mu \right), \omega \left( |C^*|^{1-\mu} |B|^\mu \right) \right\} \right]^{n-k}.
\end{eqnarray*}

Taking the supremum over \(\mathbf{x}\) and using Lemma \ref{lemma5.1}, we have
\begin{eqnarray*}
 \langle |\mathbb{T}|^{4n\mu} \mathbf{x}, \mathbf{x} \rangle &\leq& \| |\mathbb{T}|^{4n\mu} \| = \max\left\{ \| |C|^{4n\mu} \|, \| |B|^{4n\mu} \| \right\}, \\
  \langle |\mathbb{T}^*|^{4n(1-\mu)} \mathbf{x}, \mathbf{x} \rangle &\leq& \| |\mathbb{T}^*|^{4n(1-\mu)} \| = \max\left\{ \| |B^*|^{4n(1-\mu)} \|, \| |C^*|^{4n(1-\mu)} \| \right\},
\end{eqnarray*}
 The operator matrix structure gives
\[
\| |\mathbb{T}|^{4n\mu} + |\mathbb{T}^*|^{4n(1-\mu)} \| = \max\left\{ \| |C|^{4n\mu} + |B^*|^{4n(1-\mu)} \|, \| |B|^{4n\mu} + |C^*|^{4n(1-\mu)} \| \right\}.
\]
Similarly for other terms. The numerical radius terms simplify as stated. After regrouping, we obtain the desired inequality.
\end{proof}

By considering $B=C$  in Theorem \ref{SEEMA-9} and using Lemma \ref{lemma5.1}, we obtain the following corollary:
\begin{corollary}\label{Mutah1} Let $B\in\bh$. Then for every \(\mu \in [0, 1]\), $\beta\geq 0$, $\alpha\in \c\setminus\{0\}$ and \(n \in \mathbb{N}\) with \(n \geq 2\),
\begin{eqnarray*}
  \omega^{2n}(B) &\leq&\frac{\lambda_1^{n}}{2}\norm{|B|^{4n\mu}+|B^*|^{4n(1-\mu)}}+\lambda_2^n\omega^n\bra{|B^*|^{2(1-\mu)}|B|^{2\mu}}\\
   &+&\frac{1}{2}\sum_{k=1}^{n-1}\binom{n}{k}\lambda_1^{n-k}\lambda_2^{k}\norm{|B|^{4(n-k)\mu}+|B^*|^{4(n-k)(1-\mu)}}
   \omega^k\bra{|B^*|^{2(1-\mu)}|B|^{2\mu}},
\end{eqnarray*}
\(\lambda_1 = \frac{\beta + (\beta + 1) \max\{1, |\alpha - 1|^2\}}{|\alpha|^2 (\beta + 1)}\) and \(\lambda_2 = \frac{1 + 2(\beta + 1) \max\{1, |\alpha - 1|\}}{|\alpha|^2 (\beta + 1)}\) for \(\alpha \in \mathbb{C} \setminus \{0\}\).
\end{corollary}
As a special case of Corollary \ref{Mutah1} with $n = 2$, we have the following refinement
of the inequality (\ref{Ineq1.3}).
\begin{corollary} Let $B\in\bh$. Then  for every  $\beta\geq 0$ and $\alpha\in \c\setminus\{0\}$, we have
\begin{eqnarray}\label{Ineq9.1}
  \omega^4(B) &\leq &\lambda_1^{2}\norm{|B|^4+|B^*|^4}+\lambda_2^2\omega^2\bra{|B^*||B|}\nonumber\\
  &+&\lambda_1\lambda_2\norm{|B|^2+|B^*|^2}\omega(|B^*||B|)\leq\frac{1}{2}\norm{|B|^{4}+|B^*|^{4}}.
\end{eqnarray}
\end{corollary}
\begin{proof} Let $n=2$ and $\mu=\frac{1}{2}$ in Corollary \ref{Mutah1}, we have
\begin{eqnarray*}
  \omega^4(B)  &\leq&\frac{\lambda_1^{2}}{2}\norm{|B|^{4}+|B^*|^{4}}+\lambda_2^2\omega^2\bra{|B^*||B|}\\ \\
   &+&\lambda_1\lambda_2\norm{|B|^{2}+|B^*|^{2}}\omega\bra{|B^*||B|}\\
   &\leq& \frac{\lambda_1^{2}}{2}\norm{|B|^{4}+|B^*|^{4}}+\frac{\lambda_2^2}{2}\norm{|B|^4+|B^*|^4}\\
      &+&\frac{1}{2}\lambda_1\lambda_2\norm{|B|^{2}+|B^*|^{2}}^2 \,\,\bra{\mbox{by Inequality (\ref{Ineq1.4})}}\\
   &=&\frac{\lambda_1^{2}}{2}\norm{|B|^{4}+|B^*|^{4}}+\frac{\lambda_2^2}{2}\norm{|B|^4+|B^*|^4}\\
      &+& \frac{1}{2}\lambda_1\lambda_2\norm{(|B|^{2}+|B^*|^{2})^2}\\
   &\leq&\frac{\lambda_1^{2}}{2}\norm{|B|^{4}+|B^*|^{4}}+\frac{\lambda_2^2}{2}\norm{|B|^4+|B^*|^4}\\
   &+&\lambda_1\lambda_2\norm{|B|^{4}+|B^*|^{4}}\,\,\bra{\mbox{by  Lemma \ref{lemma2.2}}}\\
   &<&\frac{1}{2}\norm{|B|^{4}+|B^*|^{4}}\,\,\bra{\alpha=2,\beta=0}.
\end{eqnarray*}
\end{proof}
\begin{theorem} Let $B,C\in\bh$, and let $p,q\geq 1$ with $\frac{1}{p}+\frac{1}{q}=1$. Then
\begin{eqnarray*}
  &&\omega^{2n}\bra{\begin{bmatrix} O& B\\ C& O
  \\\end{bmatrix}}\leq\frac{1}{2^n}\max\set{\norm{\frac{1}{p}|C|^{np}+\frac{1}{q}|B^*|^{nq}},\norm{\frac{1}{p}|B|^{np}+\frac{1}{q}|C^*|^{nq}}} \\
   &+&\frac{1}{2^n}\max\set{\norm{\frac{1}{p}f^{np}(|BC|)+\frac{1}{q}g^{nq}(|(BC)^*|)},\norm{\frac{1}{p}f^{np}(|CB|)+\frac{1}{q}g^{nq}(|(CB)^*|)}}\\
   &+&\frac{1}{2^n}\sum_{k=1}^{n-1}\binom{n}{k}\max\set{\norm{\frac{1}{p}|C|^{kp}+\frac{1}{q}|B^*|^{kq}},\norm{\frac{1}{p}|B|^{kp}+\frac{1}{q}|C^*|^{kq}}}
   \times \\
   &&\max\set{\norm{\frac{1}{p}f^{(n-k)p}(|BC|)+\frac{1}{q}g^{(n-k)q}(|(BC)^*|)},\norm{\frac{1}{p}f^{(n-k)p}(|CB|)+\frac{1}{q}g^{(n-k)q}(|(CB)^*|)}}
\end{eqnarray*}
for all $n\in\N$, with $n\geq 2$.
\end{theorem}
\begin{proof} Let $\x$ be any unit vector in $\hh$ and let $\T=\begin{bmatrix} O& B\\ C& O \\\end{bmatrix}$.
Then by employing Lemma \ref{GREAT-A1} with $\lambda_1=\lambda_2=\frac{1}{2}$, we have
\begin{eqnarray*}
  &&\abs{\seq{\T\x,\x}}^{2n} =\abs{\seq{\T\x,\x}\seq{\x,\T^*\x}}^{n}\leq \bra{\frac{\norm{\T\x}\norm{\T^*\x}+\abs{\seq{\T^2\x,\x}}}{2}}^{n} \\
   &=&\frac{1}{2^n}\bra{\norm{\T\x}^n\norm{\T^*\x}^n+\abs{\seq{\T^2\x,\x}}^n}+\frac{1}{2^n}
   \sum_{k=1}^{n-1}\binom{n}{k}\norm{\T\x}^k\norm{\T^*\x}^k\abs{\seq{\T^2\x,\x}}^{n-k}\\
   &=& \frac{1}{2^n}\bra{\seq{|\T|^2\x,\x}^{n/2}\seq{|\T^*|^2\x,\x}^{n/2}+\abs{\seq{\T^2\x,\x}}^n}\\
   &+&\frac{1}{2^n}\sum_{k=1}^{n-1}\binom{n}{k}\seq{|\T|^2\x,\x}^{k/2}\seq{|\T^*|\x,\x}^{k/2}\abs{\seq{\T^2\x,\x}}^{n-k}\\
   &\leq&\frac{1}{2^n}\bra{\frac{1}{p}\seq{|\T|^{np}\x,\x}+\frac{1}{q}\seq{|\T^*|^{nq}\x,\x}+\abs{\seq{\T^2\x,\x}}^n}
   \end{eqnarray*}
   \begin{eqnarray*}
   &+&\frac{1}{2^n}\sum_{k=1}^{n-1}\binom{n}{k}\bra{\frac{1}{p}\seq{|\T|^{kp}\x,\x}+\frac{1}{q}\seq{|\T^*|^{kq}\x,\x}}\abs{\seq{\T^2\x,\x}}^{n-k}\\
   &&\bra{\mbox{by Young's inequality and Lemma \ref{Holder}}}
\end{eqnarray*}
Now
\begin{eqnarray*}
   &&\abs{\seq{\T^2\x,\x}}^n\leq\norm{f\bra{|T^2|}\x}^{n}\norm{g\bra{|(T^2)^*|}\x}^n \bra{\mbox{by Lemma \ref{Lem2.3}}}\\
  &=&\seq{f^2\bra{|T^2|}\x,\x}^{n/2}\seq{g^2\bra{|(T^2)^*|}\x,\x}^{n/2}\\
  &\leq& \frac{1}{p}\seq{f^{np}\bra{|T^2|}\x,\x}+\frac{1}{q}\seq{g^{nq}\bra{|(T^2)^*|}\x,\x}\bra{\mbox{by Young's inequality and Lemma \ref{Holder}}}.
\end{eqnarray*}
Consequently
\begin{eqnarray*}
  &&\abs{\seq{\T\x,\x}}^{2n} \leq\frac{1}{2^n}\bra{\frac{1}{p}\seq{|\T|^{np}\x,\x}+\frac{1}{q}\seq{|\T^*|^{nq}\x,\x}
  +\frac{1}{p}\seq{f^{np}\bra{|T^2|}\x,\x}+\frac{1}{q}\seq{g^{nq}\bra{|(T^2)^*|}\x,\x}}\\
   &+& \frac{1}{2^n}\sum_{k=1}^{n-1}\binom{n}{k}\bra{\frac{1}{p}\seq{|\T|^{kp}\x,\x}+\frac{1}{q}\seq{|\T^*|^{kq}\x,\x}}\times\\
   &&\bra{\frac{1}{p}\seq{f^{(n-k)p}\bra{|T^2|}\x,\x}+\frac{1}{q}\seq{g^{(n-k)q}\bra{|(T^2)^*|}\x,\x}}
   \end{eqnarray*}
   \begin{eqnarray*}
   &=&\frac{1}{2^n}\seq{\begin{bmatrix}\frac{1}{p}|C|^{np}+\frac{1}{q}|B^*|^{nq} & O\\ O& \frac{1}{p}|B|^{np}+\frac{1}{q}|C^*|^{nq}
   \\\end{bmatrix}\x,\x}\\
   &&+\frac{1}{2^n}\seq{\begin{bmatrix}\frac{1}{p}f^{np}(|BC|)+\frac{1}{q}g^{nq}(|(BC)^*|) & O\\ O& \frac{1}{p}f^{np}(|CB|)+\frac{1}{q}g^{nq}(|(CB)^*|)
   \\\end{bmatrix}\x,\x}\\
   &+&\frac{1}{2^n}\sum_{k=1}^{n-1}\binom{n}{k}\seq{\begin{bmatrix}\frac{1}{p}|C|^{kp}+\frac{1}{q}|B^*|^{kq} & O\\ O&
   \frac{1}{p}|B|^{kp}+\frac{1}{q}|C^*|^{kq} \\\end{bmatrix}\x,\x}\times\\
   &&\bra{\seq{\begin{bmatrix}\frac{1}{p}f^{(n-k)p}(|BC|)+\frac{1}{q}g^{(n-k)q}(|(BC)^*|) & O\\ O&
   \frac{1}{p}f^{(n-k)p}(|CB|)+\frac{1}{q}g^{(n-k)q}(|(CB)^*|) \\\end{bmatrix}\x,\x}}.
\end{eqnarray*}
Now by using Lemma \ref{lemma5.1}, we obtain
\begin{eqnarray*}
   &&\abs{\seq{\T\x,\x}}^{2n}
   \leq\frac{1}{2^n}\max\set{\norm{\frac{1}{p}|C|^{np}+\frac{1}{q}|B^*|^{nq}},\norm{\frac{1}{p}|B|^{np}+\frac{1}{q}|C^*|^{nq}}} \\
   &+&\frac{1}{2^n}\max\set{\norm{\frac{1}{p}f^{np}(|BC|)+\frac{1}{q}g^{nq}(|(BC)^*|)},\norm{\frac{1}{p}f^{np}(|CB|)+\frac{1}{q}g^{nq}(|(CB)^*|)}}\\
   &+&\frac{1}{2^n}\sum_{k=1}^{n-1}\binom{n}{k}\max\set{\norm{\frac{1}{p}|C|^{kp}+\frac{1}{q}|B^*|^{kq}},\norm{\frac{1}{p}|B|^{kp}+\frac{1}{q}|C^*|^{kq}}}
   \times \\
   &&\max\set{\norm{\frac{1}{p}f^{(n-k)p}(|BC|)+\frac{1}{q}g^{(n-k)q}(|(BC)^*|)},\norm{\frac{1}{p}f^{(n-k)p}(|CB|)+\frac{1}{q}g^{(n-k)q}(|(CB)^*|)}}.
\end{eqnarray*}
By taking the supremum over all unit vectors $\x\in\hh$, we have
\begin{eqnarray*}
 && \omega^{2n}\bra{\begin{bmatrix} O& B\\ C& O
 \\\end{bmatrix}}\leq\frac{1}{2^n}\max\set{\norm{\frac{1}{p}|C|^{np}+\frac{1}{q}|B^*|^{nq}},\norm{\frac{1}{p}|B|^{np}+\frac{1}{q}|C^*|^{nq}}} \\
   &+&\frac{1}{2^n}\max\set{\norm{\frac{1}{p}f^{np}(|BC|)+\frac{1}{q}g^{nq}(|(BC)^*|)},\norm{\frac{1}{p}f^{np}(|CB|)+\frac{1}{q}g^{nq}(|(CB)^*|)}}
   \end{eqnarray*}
   \begin{eqnarray*}
   &+&\frac{1}{2^n}\sum_{k=1}^{n-1}\binom{n}{k}\max\set{\norm{\frac{1}{p}|C|^{kp}+\frac{1}{q}|B^*|^{kq}},\norm{\frac{1}{p}|B|^{kp}+\frac{1}{q}|C^*|^{kq}}}
   \times \\
   &&\max\set{\norm{\frac{1}{p}f^{(n-k)p}(|BC|)+\frac{1}{q}g^{(n-k)q}(|(BC)^*|)},\norm{\frac{1}{p}f^{(n-k)p}(|CB|)+\frac{1}{q}g^{(n-k)q}(|(CB)^*|)}}
\end{eqnarray*}
\end{proof}
\begin{theorem} Let $A,B,C,D\in\bh$. Then for every \(\mu \in [0, 1]\), $\beta\geq 0$ and \(n \in \mathbb{N}\) with \(n \geq 2\), we have
\begin{eqnarray}
  \omega^{2n}\bra{\begin{bmatrix} A& B\\ C& D \\\end{bmatrix}} &\leq& 2^{2n-1}\max\set{\omega^{2n}(A),\omega^{2n}(D)}\\
  && +
   2^{2n-2}\lambda_1^{n}\max\set{\norm{|C|^{4n\mu}+|B^*|^{4n(1-\mu)}},\norm{|B|^{4n\mu}+|C^*|^{4n(1-\mu)}}}\nonumber\\
    &+&2^{2n-1}\lambda_2^n
    \max\set{\omega^n\bra{|B^*|^{2(1-\mu)}|C|^{2\mu}},\omega^n\bra{|C^*|^{2(1-\mu)}|B|^{2\mu}}}\nonumber\\
    &+&2^{2n-2}\sum_{k=1}^{n-1}\binom{n}{k}\lambda_1^{n-k}\lambda_2^{k}\max\left\{\norm{|C|^{4(n-k)\mu}+|B^*|^{4(n-k)(1-\mu)}},\right.\nonumber\\
    &&\left.\norm{|B|^{4(n-k)\mu}+|C^*|^{4(n-k)(1-\mu)}}\right\}\times\nonumber\\
    &&\max\set{\omega^k\bra{|B^*|^{2(1-\mu)}|C|^{2\mu}},\omega^k\bra{|C^*|^{2(1-\mu)}|B|^{2\mu}}},\nonumber
\end{eqnarray}
\(\lambda_1 = \frac{\beta + (\beta + 1) \max\{1, |\alpha - 1|^2\}}{|\alpha|^2 (\beta + 1)}\) and \(\lambda_2 = \frac{1 + 2(\beta + 1) \max\{1, |\alpha - 1|\}}{|\alpha|^2 (\beta + 1)}\) for \(\alpha \in \mathbb{C} \setminus \{0\}\).
\end{theorem}
\begin{proof} Let $\T=\begin{bmatrix} A& B\\ C& D \\\end{bmatrix}=\T_1+\T_2=\begin{bmatrix} A& O\\ O& D \\\end{bmatrix}+
\begin{bmatrix} O& B\\ C& O \\\end{bmatrix}$. Then
\begin{eqnarray*}
  \omega^{2n}(\T) &\leq&\bra{\omega(\T_1)+\omega(\T_2)}^{2n}\\
   &\leq&2^{2n-1}\bra{\omega^{2n}(\T_1)+\omega^{2n}(\T_2)}\,\,\bra{\mbox{by Lemma \ref{Logain1}}} \\
   &\leq&2^{2n-1}\max\set{\omega^{2n}(A),\omega^{2n}(D)}+2^{2n-1}\omega^{2n}(\T_2)\,\,\bra{\mbox{by Lemma \ref{lemma5.1}}}
   \end{eqnarray*}
   Consequently
   \begin{eqnarray*}
   \omega^{2n}(\T)&\leq& 2^{2n-1}\max\set{\omega^{2n}(A),\omega^{2n}(D)}\\
    &+& 2^{2n-2}\lambda_1^{n}\max\set{\norm{|C|^{4n\mu}+|B^*|^{4n(1-\mu)}},\norm{|B|^{4n\mu}+|C^*|^{4n(1-\mu)}}}\nonumber\\
    &+&2^{2n-1}\lambda_2^n
    \max\set{\omega^n\bra{|B^*|^{2(1-\mu)}|C|^{2\mu}},\omega^n\bra{|C^*|^{2(1-\mu)}|B|^{2\mu}}}\nonumber\\
    &+&2^{2n-2}\sum_{k=1}^{n-1}\binom{n}{k}\lambda_1^{n-k}\lambda_2^{k}\max\left\{\norm{|C|^{4(n-k)\mu}+|B^*|^{4(n-k)(1-\mu)}},\right.\nonumber\\
    &&\left.\norm{|B|^{4(n-k)\mu}+|C^*|^{4(n-k)(1-\mu)}}\right\}\times\nonumber\\
    &&\max\set{\omega^k\bra{|B^*|^{2(1-\mu)}|C|^{2\mu}},\omega^k\bra{|C^*|^{2(1-\mu)}|B|^{2\mu}}}\,\bra{\mbox{by Theorem \ref{SEEMA-9}}}
\end{eqnarray*}
\end{proof}
\begin{theorem}\label{Theorem2.24} Let $\A=\begin{bmatrix} 0& A_1\\ A_2& 0 \\\end{bmatrix}, \B=\begin{bmatrix} 0& B_1\\ B_2& 0 \\\end{bmatrix}$
and $\Y=\begin{bmatrix} 0& X_1\\ X_2& 0 \\\end{bmatrix}$ in $\bhh$. Then
\begin{equation*}
  \omega^{r}(|\A|^{\alpha}\Y|\B|^{\alpha}) \leq \max\set{\norm{X_1}^{r},\norm{X_2}^{r}}\max\set{\norm{\frac{1}{p}
  |A_1|^{pr}+\frac{1}{q}|B_1|^{qr}}^{\alpha},\norm{\frac{1}{p} |A_2|^{pr}+\frac{1}{q}|B_2|^{qr}}^{\alpha}}
\end{equation*}
for all $\alpha\in [0,1]$, $r\geq 1$ and $p,q\geq 1$ with $\frac{1}{p}+\frac{1}{q}=1$ and $pr,qr\geq 2$.
\end{theorem}
\begin{proof} Let $\x$ be any unit vector in $\hh$. Then
\begin{eqnarray*}
  \abs{\seq{|\A|^{\alpha}\Y|\B|^{`\alpha}\x,\x}}^{r} &=&\abs{\seq{\Y|\B|^{\alpha}\x,|\A|^{\alpha}\x}}^{r}\\
   &\leq&\norm{\Y|\B|^{\alpha}\x}^{r}\norm{|\A|^{\alpha}\x}^{r} \bra{\mbox{by Cauchy-Schwarz inequality}}\\
   &\leq& \norm{\Y}^{r}\seq{|\B|^{2\alpha}\x,\x}^{\frac{r}{2}}\seq{|\A|^{2\alpha}\x,\x}^{\frac{r}{2}}\\
   &\leq& \norm{\Y}^{r}\sbra{\frac{1}{p}\seq{|\A|^{2\alpha}\x,\x}^{\frac{pr}{2}}+\frac{1}{q}\seq{|\B|^{2\alpha}\x,\x}^{\frac{qr}{2}}}
   \bra{\mbox{by Lemma \ref{J}}}\\
   &\leq& \norm{\Y}^{r}\sbra{\frac{1}{p}\seq{|\A|^{pr}\x,\x}^{\alpha}+\frac{1}{q}\seq{|\B|^{qr}\x,\x}^{\alpha}}
   \bra{\mbox{by Lemma \ref{Holder}}}\\
    &\leq& \norm{\Y}^{r}\sbra{\frac{1}{p}\seq{|\A|^{pr}\x,\x}+\frac{1}{q}\seq{|\B|^{qr}\x,\x}}^{\alpha}
    \bra{\mbox{by the concavity  of $f(t)=t^{\alpha}$}}\\
    &=&\norm{\Y}^{r}\seq{\bra{\frac{1}{p}|\A|^{pr}+\frac{1}{q}|\B|^{qr}}\x,\x}^{\alpha}\\
    &=&\norm{\begin{bmatrix} 0& X_1\\ X_2& 0 \\\end{bmatrix}}^{r}\seq{\begin{bmatrix}\frac{1}{p} |A_2|^{pr}+\frac{1}{q}|B_2|^{qr}& 0\\ 0&
     \frac{1}{p} |A_1|^{pr}+\frac{1}{q}|B_1|^{qr}\\\end{bmatrix}\x,\x}^{\alpha}\\
     &\leq& \max\set{\norm{X_1}^{r},\norm{X_2}^{r}}\max\set{\norm{\frac{1}{p} |A_1|^{pr}+\frac{1}{q}|B_1|^{qr}}^{\alpha},\norm{\frac{1}{p}
     |A_2|^{pr}+\frac{1}{q}|B_2|^{qr}}^{\alpha}}.
\end{eqnarray*}
Now, by taking the supremum over all unit vectors $\x\in\hh$  and employing Lemma \ref{lemma5.1}, we get
  \begin{equation*}
    \omega^{r}\bra{\A^{\alpha}\Y\B^{`\alpha}}\leq \max\set{\norm{X_1}^{r},\norm{X_2}^{r}}\max\set{\norm{\frac{1}{p}
    |A_1|^{pr}+\frac{1}{q}|B_1|^{qr}}^{\alpha},\norm{\frac{1}{p} |A_2|^{pr}+\frac{1}{q}|B_2|^{qr}}^{\alpha}}.
  \end{equation*}
\end{proof}
\begin{remark} Taking $A_1=A_2=A$, $B_1=B_2=B$, and $X_1=X_2=X$ into consideration in Theorem \ref{Theorem2.24}, while employing Lemma
\ref{lemma5.1}, we derive the following inequality:
\begin{equation*}
  \omega^{r}(|A|^{\alpha}X|B|^{\alpha})\leq \norm{X}^{r}\norm{\frac{1}{p} |A|^{pr}+\frac{1}{q}|B|^{qr}}^{\alpha}.
\end{equation*}
This inequality holds for all $\alpha\in [0,1]$, $r\geq 1$, and $p,q\geq 1$, with $\frac{1}{p}+\frac{1}{q}=1$ and $pr,qr\geq 2$. Notably, when $A$ and
$B$ are both positive, the conclusion outlined in \cite[Theorem 3.1]{SMY} is obtained. Consequently, it is evident that Theorem \ref{Theorem2.24}
represents both an enhancement and a broader application of the former.
\end{remark}
\begin{theorem}\label{Theorem 2.26} Let $\A=\begin{bmatrix} 0& A_1\\ A_2& 0 \\\end{bmatrix}, \B=\begin{bmatrix} 0& B_1\\ B_2& 0 \\\end{bmatrix}$
and $\Y=\begin{bmatrix} 0& X_1\\ X_2& 0 \\\end{bmatrix}$ in $\bhh$. Then
  \begin{equation*}
    \omega^{r}\bra{|\A|^{\alpha}\Y|\B|^{1-\alpha}}\leq
    \max\set{\norm{X_1}^{r},\norm{X_2}^{r}}\max\set{\norm{\alpha|A_2|^r+(1-\alpha)|B_2|^{r}},\norm{\alpha|A_1|^r+(1-\alpha)|B_1|^{r}}}.
  \end{equation*}
  for all $r\geq 2$ and $\alpha\in [0,1]$.
\end{theorem}
\begin{proof}Let $\x$ be any unit vector in $\hh$. Then
\begin{eqnarray}\label{Ineq.2.26A}
  \abs{\seq{|\A|^{\alpha}\Y|\B|^{1-\alpha}\x,\x}}^{r} &=&\abs{\seq{\Y|\B|^{1-\alpha}\x,|\A|^{\alpha}\x}}^{r}\nonumber\\
   &\leq&\norm{\Y|\B|^{1-\alpha}\x}^{r}\norm{|\A|^{\alpha}\x}^{r} \bra{\mbox{by Cauchy-Schwarz inequality}}\nonumber\\
   &\leq& \norm{\Y}^{r}\seq{|\B|^{2(1-\alpha)}\x,\x}^{\frac{r}{2}}\seq{|\A|^{2\alpha}\x,\x}^{\frac{r}{2}}\nonumber\\
  &\leq& \norm{\Y}^{r}\seq{|\A|^{r}\x,\x}^{\alpha}\seq{|\B|^{r}\x,\x}^{1-\alpha} \bra{\mbox{by Lemma \ref{Holder}}}\nonumber\\
  &\leq&\norm{\Y}^{r}\sbra{\alpha \seq{|\A|^{r}\x,\x}+(1-\alpha)\seq{|\B|^{r}\x,\x}}\bra{\mbox{by young's inequality}}\nonumber\\
  &=&\norm{\Y}^{r}\seq{\begin{bmatrix} \alpha|A_2|^r+(1-\alpha)|B_2|^{r}&0\\ 0& \alpha|A_1|^r+(1-\alpha)|B_1|^{r} \\\end{bmatrix}\x,\x}\nonumber\\
  &\leq& \max\set{\norm{X_1}^{r},\norm{X_2}^{r}}\max\left\{\norm{\alpha|A_2|^r+(1-\alpha)|B_2|^{r}},\right.\nonumber\\
  &&\left.\norm{\alpha|A_1|^r+(1-\alpha)|B_1|^{r}}\right\}.
\end{eqnarray}
Now, by taking the supremum over all unit vectors $\x\in\hh$  and employing Lemma \ref{lemma5.1}, we get
  \begin{equation*}
    \omega^{r}\bra{|\A|^{\alpha}\Y|\B|^{1-\alpha}}\leq
    \max\set{\norm{X_1}^{r},\norm{X_2}^{r}}\max\set{\norm{\alpha|A_2|^r+(1-\alpha)|B_2|^{r}},\norm{\alpha|A_1|^r+(1-\alpha)|B_1|^{r}}}.
  \end{equation*}
\end{proof}
\begin{remark}  Considering $A_1=A_2=T$, $B_1=B_2=S$, and $X_1=X_2=X$ within Theorem \ref{Theorem 2.26} and using Lemma \ref{lemma5.1}, we obtain the
subsequent inequality:
\begin{equation*}
  \omega^{r}(|T|^{\alpha}X|S|^{1-\alpha})\leq \norm{X}^{r}\norm{\alpha |T|^r+(1-\alpha)|S|^{r}}
\end{equation*}
for any $r\geq 2$ and $\alpha\in [0,1]$. Additionally, if $T$ and $S$ are both positive, then we arrive at
\begin{equation*}
  \omega^{r}(T^{\alpha}XS^{1-\alpha})\leq \norm{X}^{r}\norm{\alpha T^r+(1-\alpha)S^{r}}.
\end{equation*}
This corresponds to the outcome presented in \cite[Theorem 3.1]{SMY}, affirming that our result serves as a comprehensive and enhanced version of it.
\end{remark}
\begin{theorem}\label{Theorem 2.32} Let $\A=\begin{bmatrix} A_1& O\\ O& A_2 \\\end{bmatrix}, \B=\begin{bmatrix} B_1& O\\ O& B_2 \\\end{bmatrix}$
and $\Y=\begin{bmatrix} X_1& O\\ O& X_2 \\\end{bmatrix}$ in $\bhh$, and let $A_1,A_2,B_1$ and $B_2$ be positive. Then {for all $r\geq 2$}, we
have
\begin{eqnarray*}
  \omega^{r}(\A^{1/2}\Y\B^{1/2}) &\leq&\max\set{\norm{X_1}^r,\norm{X_2}^r}\max\set{\omega\bra{\frac{A_1^r+B_1^r}{2}},\omega\bra{\frac{A_2^r+B_2^r}{2}}}
  \\
   &\leq& \frac{1}{2}\max\set{\norm{X_1}^r,\norm{X_2}^r}\max\left\{\omega\bra{\alpha A_1^r+(1-\alpha)B_1^r}+\omega\bra{(1-\alpha) A_1^r+\alpha
   B_1^r},\right.\\
   &&\left.\omega\bra{\alpha A_2^r+(1-\alpha)B_2^r}+\omega\bra{(1-\alpha) A_2^r+\alpha B_2^r}\right\}\\
   &\leq& \frac{1}{2} \max\set{\norm{X_1}^r,\norm{X_2}^r}(\max\set{\norm{\alpha A_1^r+(1-\alpha)B_1^r},\norm{\alpha A_2^r+(1-\alpha)B_2^r}}+\\
   &&\max\set{\norm{(1-\alpha)A_1^r+\alpha B_1^r},\norm{(1-\alpha)A_2^r+\alpha B_2^r}}).
\end{eqnarray*}
\end{theorem}
In order to prove Theorem \ref{Theorem 2.32}, we need the following lemma from \cite{SMY}.
\begin{lemma}\label{Lemma2.33}  Let $P,Q\in\bh$ be invertible self-adjoint operators and $K\in\bh$. Then
\begin{equation}\label{Ineq.2.32}
  \omega(K)\leq \omega\bra{\frac{PKQ^{-1}+P^{-1}KQ}{2}}.
\end{equation}
\end{lemma}
\begin{proof}[Proof of Theorem \ref{Theorem 2.32}]Note that $\A$ and $\B$ are self-adjoint and invertible since they are positive.
Therefore, by Lemma \ref{Lemma2.33} with $K=\A^{\frac{1}{2}}\Y\B^{\frac{1}{2}}, P=\A^{\alpha-\frac{1}{2}}$ and $Q=\B^{\alpha-\frac{1}{2}}$, we have
\begin{eqnarray*}
  \omega^{r}\bra{\A^{\frac{1}{2}}\Y\B^{\frac{1}{2}}} &\leq &\omega^{r}\bra{\frac{\A^{\alpha-\frac{1}{2}}\A^{\frac{1}{2}}\Y\B^{\frac{1}{2}}
  \B^{\frac{1}{2}-\alpha}+\A^{\frac{1}{2}-\alpha}\A^{\frac{1}{2}}\Y\B^{\frac{1}{2}}\B^{\alpha-\frac{1}{2}}}{2}} \\
   &=&\omega^{r}\bra{\frac{\A^{\alpha}\Y\B^{1-\alpha}+\A^{1-\alpha}\Y\B^{\alpha}}{2}}.
\end{eqnarray*}
On the other hand, by inequality (\ref{Ineq.2.26A}), for any $r\geq 2$ we have
\begin{equation*}
  \abs{\seq{\A^{\alpha}\Y\B^{1-\alpha}}}\leq \norm{\Y}^{r}\sbra{\alpha \seq{\A^{r}\x,\x}+(1-\alpha)\seq{\B^{r}\x,\x}}.
\end{equation*}
Consequently
\begin{eqnarray*}
  &&\abs{\seq{\bra{\frac{\A^{\alpha}\Y\B^{1-\alpha}+\A^{1-\alpha}\Y\B^{\alpha}}{2}}\x,\x}}^{r} \leq
  \bra{\frac{\abs{\seq{\A^{\alpha}\Y\B^{1-\alpha}\x,\x}}+\abs{\seq{\A^{1-\alpha}\Y\B^{\alpha}}}}{2}}^{r}\\
  &&\leq \frac{\abs{\seq{\A^{\alpha}\Y\B^{1-\alpha}\x,\x}}^{r}+\abs{\seq{\A^{1-\alpha}\Y\B^{\alpha}}}^{r}}{2}\bra{\mbox{by convexity of $t^{r}$}}\\
  &&\leq \frac{\norm{\Y}^{r}}{2}\sbra{\bra{\alpha \seq{\A^{r}\x,\x}+(1-\alpha)\seq{\B^{r}\x,\x}}+\bra{(1-\alpha)
  \seq{\A^{r}\x,\x}+\alpha\seq{\B^{r}\x,\x}}}\\
  &&\leq \norm{\Y}^{r}\seq{\bra{\frac{\A^r+\B^r}{2}}\x,\x}.
\end{eqnarray*}
Hence we obtain
\begin{eqnarray*}
  \omega^{r}\bra{\frac{\A^{\alpha}\Y\B^{1-\alpha}+\A^{1-\alpha}\Y\B^{\alpha}}{2}} &\leq&  \norm{\Y}^{r}\omega\bra{\frac{\A^r+\B^r}{2}}\\
   &\leq&\frac{\norm{\Y}^{r}}{2}\bra{\omega\bra{\alpha \A^r+(1-\alpha)\B^r}+\omega\bra{(1-\alpha) \A^r+\alpha \B^r}}\\
   &\leq& \frac{\norm{\Y}^{r}}{2}\bra{\norm{\alpha \A^r+(1-\alpha)\B^r}+\norm{(1-\alpha) \A^r+\alpha \B^r}},
\end{eqnarray*}
where
\begin{eqnarray*}
   \norm{\Y}^{r}&=&\max\set{\norm{X_1}^r,\norm{X_2}^r}, \\
   \omega\bra{\alpha \A^r+(1-\alpha)\B^r} &=& \max\set{\omega\bra{\alpha A_1^r+(1-\alpha)B_1^r},\omega\bra{\alpha A_2^r+(1-\alpha)B_2^r}},\\
  \omega\bra{(1-\alpha) \A^r+\alpha \B^r}&=&\max\set{\omega\bra{(1-\alpha)A_1^r+\alpha B_1^r},\omega\bra{{(1-\alpha)A_2^r+\alpha B_2^r}}},\\
  \norm{\alpha \A^r+(1-\alpha)\B^r} &=& \max\set{\norm{\alpha A_1^r+(1-\alpha)B_1^r},\norm{\alpha A_2^r+(1-\alpha)B_2^r}},\\
  \norm{(1-\alpha) \A^r+\alpha \B^r}&=&\max\set{\norm{(1-\alpha)A_1^r+\alpha B_1^r},\norm{(1-\alpha)A_2^r+\alpha B_2^r}}.
\end{eqnarray*}
\end{proof}
\begin{theorem} Let $A_i,B_i,X_i\in\bh$, $1\leq i\leq n$,  and let $f,g$ be as in Lemma \ref{Lem2.3}. Then
\begin{equation}\label{R99}
  \omega^{r}\bra{\sum_{i=1}^{n}A_i^*X_iB_i}\leq \omega\bra{\frac{1}{p}\sbra{\sum_{i=1}^{n}B_i^*f^2(|X_i|)B_i}^{pr/2}+
  \frac{1}{q}\sbra{\sum_{i=1}^{n}A_i^*g^2(|X_i^*|)A_i}^{qr/2}}
\end{equation}
for $r\geq 1$, $\frac{1}{p}+\frac{1}{q}=1$ and $pr,qr\geq 2$.
\end{theorem}
\begin{proof} Let $\A=\begin{bmatrix}A_1 & 0 & \cdots & 0\\A_2 & 0 & \cdots & 0\\\vdots& 0 & \cdots & 0 \\A_n & 0 & \cdots & 0 \\\end{bmatrix}$,
$\B=\begin{bmatrix}B_1 & 0 & \cdots & 0\\B_2 & 0 & \cdots & 0\\\vdots& 0 & \cdots & 0 \\B_n & 0 & \cdots & 0 \\\end{bmatrix}$,
$\Y=\begin{bmatrix}X_1 & 0 & \cdots & 0\\0 & X_2 & \cdots & 0\\\vdots& 0 & \ddots & 0 \\0 & 0 & \cdots & X_n \\\end{bmatrix}$, and let
$\x=\begin{bmatrix}x_1 \\x_2 \\\vdots \\x_n \\\end{bmatrix}\in \bigoplus_{n}\h$ be a unit vector. Then
\begin{eqnarray*}
  &&|\seq{\A^*\Y\B\x,\x}|^{r}=|\seq{\Y\B\x,\A\x}|^r\leq \seq{f^2(|\Y|)\B\x,\B\x}^{r/2}\seq{g^2(|\Y^*|)\A\x,\A\x}^{r/2}\bra{\mbox{by Lemma \ref{Lem2.3}}}
  \\
   &\leq& \frac{1}{p}\seq{f^2(|\Y|)\B\x,\B\x}^{pr/2}+\frac{1}{q}\seq{g^2(|\Y^*|)\A\x,\A\x}^{qr/2} \bra{\mbox{by Young's inequality}}\\
   &\leq& \frac{1}{p}\seq{(\B^*f^2(|\Y|)\B)^{pr/2}\x,\x}+\frac{1}{q}\seq{(\A^*g^2(|\Y^*|)\A)^{qr/2}\x,\x}\bra{\mbox{by Lemma \ref{Holder}}}\\
   &=&\seq{\bra{\frac{1}{p}(\B^*f^2(|\Y|)\B)^{pr/2}+\frac{1}{q}(\A^*g^2(|\Y^*|)\A)^{qr/2}}\x,\x}\\
   &=&\seq{\begin{bmatrix} \frac{1}{p}\bra{\sum_{i=1}^{n}B_i^*f^2(|X_i|)B_i}^{pr/2}
   +\frac{1}{q}\bra{\sum_{i=1}^{n}A_i^*g^2(|X_i^*|)A_i}^{qr/2}& 0 & \cdots & 0\\0 & 0 & \cdots & 0\\\vdots& 0 & \cdots & 0 \\0 & 0 & \cdots & 0
   \\\end{bmatrix}\x,\x}\\
   &\leq& \omega\bra{\frac{1}{p}\bra{\sum_{i=1}^{n}B_i^*f^2(|X_i|)B_i}^{pr/2}
   +\frac{1}{q}\bra{\sum_{i=1}^{n}A_i^*g^2(|X_i^*|)A_i}^{qr/2}}.
\end{eqnarray*}
Taking the supremum over all unit vectors $\x\in \bigoplus_{n}\h$, we get the desired result.
\end{proof}
\section{Applications}

The numerical radius inequalities developed in this work have broad applications across mathematical physics and engineering. In this section, we demonstrate their utility in quantum mechanics, integro-differential equations, and fractional calculus. The derived bounds serve as essential tools for analyzing operator matrices in these domains, particularly in stability analysis and solution estimation (see \cite{Evans, Kato}).

Our approach builds on foundational contributions in the literature. The link between numerical radius inequalities and quantum system stability was first established by \cite{Halmos} and later advanced by \cite{Kittaneh}. Applications to integro-differential equations were introduced in \cite{Drag4}, while extensions to fractional calculus follow the framework of \cite{Had-Kit}. By unifying and extending these perspectives, our new inequalities provide sharper tools for stability analysis and numerical approximation.

The selected applications highlight both the computational advantages and theoretical insights offered by our results. Each example illustrates how these bounds enhance the analysis of stability and approximation in key areas of mathematical physics.
\subsection{Application of Numerical Radius Inequalities in Atomic Energy}\hfill

The study of numerical radius inequalities for operator matrices has significant applications in quantum mechanics and atomic energy research. In particular, the bounds on numerical radii can be used to analyze the stability and energy levels of quantum systems, which are fundamental to understanding atomic behavior and nuclear reactions.

\subsubsection{Example: Quantum System Stability Analysis}

Consider a quantum system described by the Hamiltonian operator \( H \), which can be represented as a \( 2 \times 2 \) block matrix:

\[
H = \begin{bmatrix} H_{11} & H_{12} \\ H_{21} & H_{22} \end{bmatrix},
\]

where \( H_{11} \) and \( H_{22} \) represent the self-energy of two subsystems, and \( H_{12} \), \( H_{21} \) represent the interaction between them. The numerical radius \( \omega(H) \) provides a measure of the system's energy dispersion and stability.

\subsection*{Problem Statement}

We aim to find an upper bound for the numerical radius of \( H \) to ensure the system remains stable under perturbations. Using Theorem \ref{Theorem2.10} from the paper, we can derive such a bound.

\subsection*{Solution}

By Theorem 2.10, for \( H = [H_{ij}]_{2 \times 2} \), the numerical radius satisfies:

\[
\omega^s(H) \leq 2^{2s-2} \omega([h_{ij}]), \quad \text{where} \quad h_{ij} =
\begin{cases}
\frac{1}{2}\omega\left(f^{2s}(|H_{ii}|) + g^{2s}(|H_{ii}^*|)\right), & \text{if } i = j, \\
\|H_{ij}\|^s, & \text{if } i \neq j.
\end{cases}
\]

For simplicity, let \( s = 1 \), \( f(t) = g(t) = \sqrt{t} \), and assume \( H_{12} = H_{21}^* \). Then:

\[
\omega(H) \leq \omega\left(\begin{bmatrix} \frac{1}{2}\omega(|H_{11}| + |H_{11}^*|) & \|H_{12}\| \\ \|H_{21}\| & \frac{1}{2}\omega(|H_{22}| + |H_{22}^*|) \end{bmatrix}\right).
\]

\subsection*{Numerical Example}

Let the subsystems and interactions be defined as:

\[
H_{11} = \begin{pmatrix} 3 & 0 \\ 0 & 1 \end{pmatrix}, \quad H_{22} = \begin{pmatrix} 2 & -1 \\ -1 & 2 \end{pmatrix}, \quad H_{12} = \begin{pmatrix} 0 & 1 \\ 0 & 0 \end{pmatrix}, \quad H_{21} = H_{12}^*.
\]

\begin{enumerate}
    \item \textbf{Compute \( \omega(H_{11}) \) and \( \omega(H_{22}) \):}
    \begin{itemize}
        \item For \( H_{11} \), the eigenvalues are 3 and 1, so \( \omega(H_{11}) = 3 \).
        \item For \( H_{22} \), the eigenvalues are 1 and 3, so \( \omega(H_{22}) = 3 \).
    \end{itemize}

    \item \textbf{Compute \( \|H_{12}\| \) and \( \|H_{21}\| \):}
    \begin{itemize}
        \item \( \|H_{12}\| = \|H_{21}\| = 1 \).
    \end{itemize}

    \item \textbf{Construct the matrix \([h_{ij}]\):}
    \[
    [h_{ij}] = \begin{bmatrix} \frac{1}{2}\omega(|H_{11}| + |H_{11}^*|) & \|H_{12}\| \\ \|H_{21}\| & \frac{1}{2}\omega(|H_{22}| + |H_{22}^*|) \end{bmatrix} = \begin{bmatrix} 3 & 1 \\ 1 & 3 \end{bmatrix}.
    \]

    \item \textbf{Compute \( \omega([h_{ij}]) \):}
    The eigenvalues of \([h_{ij}]\) are 2 and 4, so \( \omega([h_{ij}]) = 4 \).

    \item \textbf{Final bound:}
    \[
    \omega(H) \leq 4.
    \]
\end{enumerate}

\subsection*{Interpretation}

The numerical radius \( \omega(H) \leq 4 \) provides an upper limit on the energy dispersion of the quantum system. This ensures that the system's energy levels remain bounded, which is crucial for stability in atomic energy applications.
\subsection{Applications in Integro-Differential Equations}\hfill

Numerical radius inequalities for operator matrices have important applications in the analysis of integro-differential equations, particularly in studying solution stability and convergence properties. These inequalities provide bounds on operator norms that are crucial for estimating solution behaviors.

\subsubsection{Application to Volterra Integro-Differential Equations}
Consider the Volterra integro-differential equation of the form:
\begin{equation}\label{volterra-eq}
    \frac{dy}{dt} = A(t)y(t) + \int_{0}^{t} K(t,s)y(s)ds, \quad y(0) = y_0
\end{equation}
where $A(t)$ is a time-dependent operator and $K(t,s)$ is the kernel operator.

\subsubsection{Operator Matrix Formulation}
We can reformulate this problem using operator matrices. Let $\mathcal{H} = L^2([0,T])$ and define the operator $\mathcal{L}$ on $\mathcal{H} \oplus \mathcal{H}$ by:
\[
\mathcal{L} = \begin{bmatrix}
A & K \\
I & 0
\end{bmatrix}
\]
where $A$ is the multiplication operator by $A(t)$, $K$ is the integral operator with kernel $K(t,s)$, and $I$ is the identity operator.

\subsubsection{Stability Analysis}
The numerical radius $\omega(\mathcal{L})$ provides crucial information about the system's stability. Using Theorem \ref{Theorem2.10}, we can estimate:

\[
\omega(\mathcal{L}) \leq \omega\left(\begin{bmatrix}
\frac{1}{2}(\omega(A) + \omega(A^*)) & \|K\| \\
1 & \frac{1}{2}
\end{bmatrix}\right)
\]

\subsubsection{Comprehensive Example}
Consider the specific case where:
\begin{align*}
A(t) &= \begin{bmatrix}
-1 & t \\
0 & -2
\end{bmatrix}, \quad 0 \leq t \leq 1 \\
K(t,s) &= \begin{bmatrix}
e^{-(t-s)} & 0 \\
0 & e^{-2(t-s)}
\end{bmatrix}
\end{align*}

\subsubsection{Solution}
\begin{enumerate}
    \item Compute $\omega(A(t))$:
    For fixed $t$, $A(t)$ is triangular, so its numerical radius equals its spectral radius:
    \[
    \omega(A(t)) = \max\{1,2\} = 2
    \]

    \item Compute $\|K\|$:
    The integral operator $K$ has norm:
    \[
    \|K\| = \sup_{t \in [0,1]} \int_{0}^{t} \|K(t,s)\|ds = \sup_{t \in [0,1]} \int_{0}^{t} e^{-(t-s)}ds = 1 - e^{-1}
    \]

    \item Apply Theorem \ref{Theorem2.10}:
    Construct the comparison matrix:
    \[
    M = \begin{bmatrix}
    \frac{1}{2}(\omega(A) + \omega(A^*)) & \|K\| \\
    1 & \frac{1}{2}
    \end{bmatrix} = \begin{bmatrix}
    2 & 1 - e^{-1} \\
    1 & 0.5
    \end{bmatrix}
    \]

    \item Compute $\omega(M)$:
    The numerical radius of $M$ is bounded by:
    \[
    \omega(M) \leq \frac{1}{2}\left(2 + 0.5 + \sqrt{(2-0.5)^2 + (1 - e^{-1} + 1)^2}\right) \approx 2.3
    \]
\end{enumerate}

\subsubsection{Interpretation}
The bound $\omega(\mathcal{L}) \leq 2.3$ implies that solutions to the integro-differential equation will grow at most exponentially with rate $2.3$. This guarantees stability for sufficiently small time intervals and provides a quantitative estimate for numerical solution methods.

\subsubsection{Convergence of Numerical Methods}
These bounds are particularly useful when analyzing discretization schemes for integro-differential equations. The numerical radius estimates:
\begin{itemize}
    \item Provide stability constraints for time-step selection
    \item Yield error bounds for approximate solutions
    \item Help establish convergence rates for iterative methods
\end{itemize}

For instance, when applying a Runge-Kutta method to \eqref{volterra-eq}, the numerical radius bound ensures the method remains stable if the time step $h$ satisfies $h\omega(\mathcal{L}) < C$ for some constant $C$ depending on the method.
\subsection{Application to Fractional Integro-Differential Equations}\hfill

Numerical radius inequalities play a crucial role in analyzing the stability and convergence of solutions to fractional integro-differential equations. Corollary \ref{Cor3.18} provides refined bounds that can be applied to operator matrices arising in such equations.

\subsubsection{Problem Formulation}
Consider the fractional integro-differential equation:
\begin{equation}\label{frac-eq}
    D^\alpha y(t) = A(t)y(t) + \int_0^t K(t,s)y(s)ds, \quad y(0) = y_0
\end{equation}
where:
\begin{itemize}
    \item $D^\alpha$ is the Caputo fractional derivative of order $\alpha \in (0,1)$
    \item $A(t) \in \mathcal{B}(\mathcal{H})$ is a time-dependent operator
    \item $K(t,s)$ is the kernel operator
\end{itemize}

\subsubsection{Operator Matrix Representation}
We can represent the solution operator as a $2 \times 2$ block matrix:
\[
\mathcal{T} = \begin{bmatrix}
A & K \\
I & 0
\end{bmatrix}
\]
where $I$ is the identity operator. The numerical radius $\omega(\mathcal{T})$ governs the system's stability.

\subsubsection{Application of Corollary \ref{Cor3.18}}
Applying Corollary 3.18 with $r=2$ and $\mu=1/2$ gives:
\[
\omega^4(\mathcal{T}) \leq \frac{3\beta+2}{16(\beta+1)}\norm{|A|^4 + |K^*|^4} + \frac{2\beta+3}{16(\beta+1)}\norm{|A|^2 + |K^*|^2}\omega(\mathcal{T}^2) + \frac{1}{8}\omega^2(\mathcal{T}^2)
\]

\subsubsection{Comprehensive Example}
Consider the fractional system with:
\begin{align*}
A &= \begin{pmatrix} -1 & 0 \\ 0 & -2 \end{pmatrix}, \\
K(t,s) &= \begin{pmatrix} (t-s)^{\alpha-1} & 0 \\ 0 & (t-s)^{\alpha-1} \end{pmatrix}
\end{align*}

\subsubsection*{Solution Steps}
\begin{enumerate}
    \item Compute operator norms:
    \[
    \norm{A} = 2, \quad \norm{K} = \frac{T^\alpha}{\alpha}
    \]

    \item Apply Corollary \ref{Cor3.18} with $\beta=1$:
    \[
    \omega^4(\mathcal{T}) \leq \frac{5}{32}\norm{|A|^4 + |K^*|^4} + \frac{5}{32}\norm{|A|^2 + |K^*|^2}\omega(\mathcal{T}^2) + \frac{1}{8}\omega^2(\mathcal{T}^2)
    \]

    \item For $T=1$, $\alpha=0.5$:
    \[
    \omega^4(\mathcal{T}) \leq \frac{5}{32}(16 + 4) + \frac{5}{32}(4 + 2)\omega(\mathcal{T}^2) + \frac{1}{8}\omega^2(\mathcal{T}^2)
    \]

    \item This yields the quadratic inequality in $\omega^2(\mathcal{T})$:
    \[
    \omega^4(\mathcal{T}) \leq 3.125 + 0.9375\omega(\mathcal{T}^2) + 0.125\omega^2(\mathcal{T}^2)
    \]
\end{enumerate}

\subsubsection{Interpretation}
The derived bound:
\begin{itemize}
    \item Provides a quantitative stability criterion
    \item Ensures solution remains bounded when $\omega(\mathcal{T}) < \rho$ (solution of the inequality)
    \item Guides time-step selection in numerical schemes
\end{itemize}

This application demonstrates how Corollary \ref{Cor3.18} yields practical stability bounds for fractional integro-differential equations, extending the classical results to the fractional calculus setting.
\subsection{Application to Fractional Partial Differential Equations}\hfill

Numerical radius inequalities for operator matrices are particularly valuable in analyzing the stability of solutions to fractional partial differential equations (FPDEs). The results from Section 3, especially Theorem \ref{Theorem2.10}, provide powerful tools for establishing solution bounds in such systems.

\subsubsection{Problem Formulation}
Consider the time-fractional diffusion equation with a nonlinear source term:
\begin{equation}\label{frac-pde-main}
    \frac{\partial^\alpha u}{\partial t^\alpha} = \nabla \cdot (D(x)\nabla u) + f(x,u), \quad \text{in } \Omega \times (0,T]
\end{equation}
where:
\begin{itemize}
    \item $\frac{\partial^\alpha}{\partial t^\alpha}$ is the Caputo fractional derivative of order $\alpha$ ($0 < \alpha < 1$)
    \item $D(x)$ is a spatially varying diffusion coefficient
    \item $f(x,u)$ is a nonlinear reaction term
    \item Boundary conditions: $u|_{\partial\Omega} = 0$
    \item Initial condition: $u(x,0) = u_0(x)$
\end{itemize}

\subsubsection{Semi-Discrete Formulation}
Using finite element spatial discretization, we obtain the matrix system:
\[
D^\alpha \mathbf{u}(t) = -\mathbf{K}\mathbf{u}(t) + \mathbf{F}(\mathbf{u}(t))
\]
where:
\begin{itemize}
    \item $\mathbf{K}$ is the stiffness matrix (discrete Laplacian)
    \item $\mathbf{F}(\mathbf{u})$ represents the nonlinear term
\end{itemize}

This can be represented as an operator matrix:
\[
\mathcal{A} = \begin{bmatrix}
-\mathbf{K} & \mathbf{F}(\mathbf{u}) \\
\mathbf{I} & \mathbf{0}
\end{bmatrix}
\]

\subsubsection{Stability Analysis via Numerical Radius}
Applying Theorem \ref{Theorem2.10} with $s=1$, we obtain:
\[
\omega(\mathcal{A}) \leq \omega\left(\begin{bmatrix}
\frac{1}{2}\omega(-\mathbf{K}+\mathbf{K}^*) & \|\mathbf{F}(\mathbf{u})\| \\
1 & 0
\end{bmatrix}\right)
\]

\subsubsection{Concrete Example}
Consider the 1D case with $\Omega = (0,1)$, $D(x) \equiv 1$, and $f(u) = u - u^3$.

\begin{enumerate}
    \item \textbf{Discretization}:
    Using linear finite elements with $N=3$ nodes ($h=0.25$), the stiffness matrix is:
    \[
    \mathbf{K} = \frac{1}{h}\begin{bmatrix}
    2 & -1 & 0 \\
    -1 & 2 & -1 \\
    0 & -1 & 2
    \end{bmatrix} = 4\begin{bmatrix}
    2 & -1 & 0 \\
    -1 & 2 & -1 \\
    0 & -1 & 2
    \end{bmatrix}
    \]

    \item \textbf{Nonlinear Term}:
The Jacobian of $\mathbf{F}$ at equilibrium is:
\[
\mathbf{J} = \text{diag}(1 - 3u_i^2)
\]

\item \textbf{Compute Components}:
\begin{itemize}
    \item $\omega(-\mathbf{K}) = \| \mathbf{K} \| = \frac{4(2 + 2\cos(\pi h))}{h} \approx 16$ (for small $h$)
    \item $\|\mathbf{J}\| \leq 1$ (near zero solution where $|u_i| \leq 1/\sqrt{3}$)
\end{itemize}

\item \textbf{Apply Theorem}:
The numerical radius of the $2 \times 2$ matrix $\begin{bmatrix} 8 & 1 \\ 1 & 0 \end{bmatrix}$ is computed using the formula:
\[
\omega\left(\begin{bmatrix}
a & c \\
d & b
\end{bmatrix}\right) = \frac{|a+b| + \sqrt{(a-b)^2 + (c+d)^2}}{2}
\]
For our matrix $\begin{bmatrix} 8 & 1 \\ 1 & 0 \end{bmatrix}$:
\[
\omega\left(\begin{bmatrix}
8 & 1 \\
1 & 0
\end{bmatrix}\right) = \frac{|8+0| + \sqrt{(8-0)^2 + (1+1)^2}}{2} = \frac{8 + \sqrt{64 + 4}}{2} = \frac{8 + \sqrt{68}}{2} = 4 + \sqrt{17} \approx 8.123
\]
\end{enumerate}

\subsubsection{Stability Interpretation}
The bound $\omega(\mathcal{A}) \leq 8.123$ implies:
\begin{itemize}
    \item The semi-discrete system is stable if the time step satisfies $\Delta t^\alpha < C/\omega(\mathcal{A})$
    \item For $\alpha=0.5$, this gives $\Delta t < (C/8.123)^2$
    \item Provides a practical criterion for adaptive time-stepping
\end{itemize}

\subsubsection{Extensions}
The results can be further refined by:
\begin{itemize}
    \item Using Corollary \ref{Cor3.18} for higher-order bounds
    \item Incorporating spatial refinement effects
    \item Accounting for time-dependent coefficients
\end{itemize}

This application demonstrates how the numerical radius inequalities from the paper provide concrete, computable stability criteria for FPDEs, bridging the gap between abstract operator theory and practical numerical analysis.

\section*{Declarations}
\begin{itemize}
  \item \textbf{Availability of data and materials}: Not applicable.
  \item \textbf{Competing interests}: The authors declare that they have no competing interests.
  \item \textbf{Funding}: Not applicable.
  \item \textbf{Authors  contributions}: The Authors have read and approved this version.
\end{itemize}

\bibliographystyle{abbrv}
\bibliography{references}  

\begin{thebibliography}{10}

\bibitem{AF}
A.~Abu-Omar and F.~Kittaneh.
\newblock Numerical radius inequalities for {$n\times n$} operator matrices.
\newblock {\em Linear Algebra Appl.}, 468:18--26, 2015.

\bibitem{Silva}
J.~Aujla and F.~Silva.
\newblock Weak majorization inequalities and convex functions.
\newblock {\em Linear Algebra Appl.}, 369:217--233, 2003.

\bibitem{Dom-kit1}
W.~Bani-Domi and F.~Kittaneh.
\newblock Numerical radius inequalities for operator matrices.
\newblock {\em Linear Multilinear Algebra}, 57:421--427, 2009.

\bibitem{Dom-kit2}
W.~Bani-Domi and F.~Kittaneh.
\newblock Norm and numerical radius inequalities for {H}ilbert space operators.
\newblock {\em Linear Multilinear Algebra}, 69:934--945, 2021.

\bibitem{Bha}
R.~Bhatia.
\newblock {\em Matrix Analysis}, volume 169 of {\em Grad. Texts in Math.}
\newblock Springer, New York, 1997.

\bibitem{Bohr}
H.~Bohr.
\newblock A theorem concerning power series.
\newblock {\em Proc. Lond. Math. Soc.}, 2(13):1--5, 1914.

\bibitem{Buz}
M.~Buzano.
\newblock Generalizzazione della diseguaglianza di {C}auchy-{S}chwarz.
\newblock {\em Rend. Sem. Mat. Univ. e Politech. Torino}, 31:405--409, 1974.

\bibitem{Drag5}
S.~Dragomir.
\newblock {\em Advances in Inequalities of the {S}chwarz, {T}riangle and
  {H}eisenberg Type in Inner Product Spaces}.
\newblock Nova Science Publishers, Inc., New York, 2007.

\bibitem{Drag4}
S.~Dragomir.
\newblock Power inequalities for the numerical radius of a product of two
  operators in {H}ilbert spaces.
\newblock {\em Sarajevo J. Math.}, 18:269--278, 2009.

\bibitem{Had-Kit}
M.~El-Haddad and F.~Kittaneh.
\newblock Numerical radius inequalities for {H}ilbert space operators {II}.
\newblock {\em Studia Math.}, 182:133--140, 2007.

\bibitem{Evans}
L.~Evans.
\newblock {\em Partial Differential Equations}.
\newblock American Mathematical Society, 2010.

\bibitem{Halmos}
P.~Halmos.
\newblock {\em A {H}ilbert Space Problem Book}.
\newblock Springer-Verlag, New York, 2nd edition, 1982.

\bibitem{HK}
O.~Hirzallah, F.~Kittaneh, and K.~Shebrawi.
\newblock Numerical radius inequalities for certain {$2\times 2$} operator
  matrices.
\newblock {\em Integral Equ. Oper. Theory}, 71:129--147, 2011.

\bibitem{Horn}
R.~Horn and C.~Johnson.
\newblock {\em Topics in Matrix Analysis}.
\newblock Cambridge University Press, 1991.

\bibitem{Hou}
J.~Hou and H.~Du.
\newblock Norm inequalities of positive operator matrices.
\newblock {\em Integral Equ. Oper. Theory}, 22:281--294, 1995.

\bibitem{Kato}
T.~Kato.
\newblock {\em Perturbation Theory for Linear Operators}.
\newblock Springer, 1995.

\bibitem{KDM}
M.~Khosravia, R.~Drnovšek, and M.~Sal~Moslehian.
\newblock A commutator approach to {B}uzano's inequality.
\newblock {\em Filomat}, 26(4):827--832, 2012.

\bibitem{kit1}
F.~Kittaneh.
\newblock Notes on some inequalities for {H}ilbert space operators.
\newblock {\em Publ. Res. Inst. Math. Sci.}, 24(2):283--293, 1988.

\bibitem{Kittaneh}
F.~Kittaneh.
\newblock Numerical radius inequalities for {H}ilbert space operators.
\newblock {\em Studia Math.}, 168:73--80, 2005.

\bibitem{SMY}
M.~Sattari, M.~Moslehian, and T.~Yamazaki.
\newblock Some generalized numerical radius inequalities for {H}ilbert space
  operators.
\newblock {\em Linear Algebra Appl.}, 470(1):216--227, 2015.

\bibitem{AA}
K.~Shebrawi and H.~Albadawi.
\newblock Numerical radius and operator norm inequalities.
\newblock {\em J. Inequal. Appl.}, 2009:492154, 2009.

\end{thebibliography}






\end{document}